\documentclass[twoside,a4paper,12pt,centertags]{amsart}

\usepackage[T1]{fontenc}
\usepackage[utf8]{inputenc}
\usepackage{subfiles}
\usepackage{multicol}

\makeatletter
\usepackage[section]{placeins}

\AtBeginDocument{%
	\expandafter\renewcommand\expandafter\subsection\expandafter{%
		\expandafter\@fb@secFB\subsection
	}%
}
\makeatother
\makeatletter
\@namedef{subjclassname@2020}{%
	\textup{2020} Mathematics Subject Classification}
\makeatother

\usepackage{amsfonts,amssymb,amsmath,amsthm}
\usepackage{mathrsfs}
\usepackage{nicefrac}
\usepackage{latexsym}
\usepackage{mathtools}
\usepackage[noadjust]{cite}
\usepackage{paralist}
\usepackage{longtable}
\usepackage{float}
\usepackage{pdfsync}
\usepackage{tikz,tikz-cd}
\usetikzlibrary{patterns}
\usepackage{hhline}
\usepackage{import}

\usepackage{enumitem}

\usepackage[colorlinks=true,citecolor=magenta,linkcolor=cyan,backref]{hyperref}
\AtBeginDocument{}

\numberwithin{equation}{section}

\theoremstyle{plain}
\newtheorem{thm}{Theorem}[section]
\newtheorem{lem}[thm]{Lemma}
\newtheorem{cor}[thm]{Corollary}
\newtheorem{prop}[thm]{Proposition}

\theoremstyle{definition}
\newtheorem{defn}[thm]{Definition}
\newtheorem{rem}[thm]{Remark}

\usepackage{color}

\DeclareMathOperator*{\esssup}{ess\,sup}

\renewcommand{\div}{\operatorname{div}}

\newcommand{\IC}{\mathbb{C}}

\newcommand{\R}{\mathbb{R}}

\newcommand{\IZ}{\mathbb{Z}}
\newcommand{\cA}{\mathcal{A}}
\newcommand{\cT}{\mathcal{T}}
\newcommand{\cD}{\mathcal{D}} 
\newcommand{\cF}{\mathcal{F}} 

\newcommand{\cK}{\mathcal{K}} 
\newcommand{\cL}{\mathcal{L}}

\newcommand{\cS}{\mathcal{S}} 

\newcommand{\loc}{\operatorname{loc}}
\renewcommand{\L}{\operatorname{L}} 
\newcommand{\C}{\operatorname{C}} 
\renewcommand{\H}{\operatorname{H}} 
\newcommand{\W}{\operatorname{W}}
  \newcommand{\I}{\operatorname{I}}
   
   \newcommand{\U}{\operatorname{U}}
\DeclareRobustCommand{\Hdot}{\dot{\H}\protect{\vphantom{H}}
} 
\DeclareRobustCommand{\Wdot}{\dot{\W}\protect{\vphantom{W}}} 
 
\renewcommand\Re{\operatorname{Re}}

\def\angle#1#2{\langle #1,#2 \rangle} 

\def\Ga{G}
\def\tGa{\widetilde G}

\def\tpsi{\tilde\psi}
\def\tf{\tilde f}

\def\tv{\tilde\phi}
\def\v{\phi}
\def\e{\mathrm{e}}

\renewcommand{\iint}{\int_{}\kern-.34em \int} 
\renewcommand{\iiint}{\iint_{}\kern-.34em \int} 

\newcommand{\dd}{\, \mathrm{d} }
\newcommand{\dv}{\, \mathrm{d} v}
\newcommand{\dx}{\, \mathrm{d} x}
\newcommand{\ds}{\, \mathrm{d} s}

\newcommand{\dt}{\, \mathrm{d} t}

\newcommand{\pascal}[1]{\textcolor{black}{#1}}

\title[Fundamental solutions to Kolmogorov-Fokker-Planck equations]{Fundamental solutions to Kolmogorov-Fokker-Planck equations with rough coefficients:  existence, uniqueness, upper estimates}
\author{Pascal Auscher}
\address{Universit\'e Paris-Saclay, CNRS, Laboratoire de Math\'{e}matiques d'Orsay, 91405 Orsay, France}
\email{pascal.auscher@universite-paris-saclay.fr}
\author{Cyril Imbert}
\address{Département de Mathématiques et Applications, École normale supérieure, Université PSL, CNRS, 75005 Paris, France}
\email{cyril.imbert@ens.psl.eu}
\author{Lukas Niebel}
\address{Institut f\"ur Analysis und Numerik,  Westf\"alische Wilhelms-Universit\"at M\"unster\\
Orl\'eans-Ring 10, 48149 M\"unster, Germany.}
\email{lukas.niebel@uni-muenster.de}
\date{November 22, 2024}

\thanks{The third author is funded by the Deutsche Forschungsgemeinschaft (DFG, German Research Foundation) under Germany's Excellence Strategy EXC 2044 --390685587, Mathematics M\"unster: Dynamics--Geometry--Structure.}



\begin{document}

\begin{abstract} We show the existence and uniqueness of fundamental solution operators to Kolmo\-gorov-Fokker-Planck equations with rough (measurable) coefficients and local or integral diffusion on finite and infinite time strips. In the local case, that is to say when the diffusion operator is of differential type, we prove $\L^2$ decay using Davies' method and the conservation property. We also prove that the existence of a generalized fundamental solution with the expected pointwise Gaussian upper bound is equivalent to Moser's $\L^2-\L^\infty$ estimates for local weak solutions to the equation and its adjoint.  When coefficients are real, this gives the existence and uniqueness of such a generalized fundamental solution and a new and natural way to obtain pointwise decay.

\end{abstract}

\maketitle

\setcounter{tocdepth}{1}
\tableofcontents

\section{Introduction}

In a previous work \cite{AIN}, through a study of embeddings for appropriate weak kinetic spaces, we provided a full theory for existence and uniqueness of weak solutions to the kinetic Cauchy problem for Kolmogorov-Fokker-Planck equations with local or integral diffusion and rough (measurable) coefficients of the variables $(t,x,v) \in I\times \R^d\times \R^d$, $I$ being a finite  or infinite time interval $(0,T), (0,\infty)$ or $\R$. The equation is of the form 
\begin{equation}
  \label{eq:weaksol}
 (\partial_t + v \cdot \nabla_x ) f +\cA f = S, 
\end{equation}
 and there is  a square integrable (in $x,v$) initial data $\psi$ at time $0$  except when  $I=\R$. We allowed $S$ to be in large classes of source spaces.
We review definitions,  assumptions and our results in Section~\ref{sec:review}, see in particular Theorem~\ref{thm:CP0T}.\\

In the first part of this article (Section~\ref{sec:construction}), we show that our results give rise to the existence and uniqueness of an evolution family $\Gamma (t,s)$ of bounded operators on $\L^2_{x,v}$ playing the role  of a fundamental solution at operator level. This is our first main result,  Theorem~\ref{thm:exunfundsol}. It goes through a synthetic and workable definition of what should be fundamental solution operators, to be distinguished from the definition of Green operators (also called propagators in the literature). Green operators are the operators solving the homogeneous initial value problem. Fundamental solution operators are the ones representing weak solutions to inhomogeneous Cauchy problems with zero initial data.  The main observation is that they must coincide and that  weak solutions at all times $t\ge 0$ are given by
\begin{align}
\label{eq:superposition}
 f(t)=\Gamma(t,0)\psi+ \int_{0}^t \Gamma(t,s) S(s)\ds,
 \end{align}
 which is in accordance with the superposition principle in physics. We mention at this stage that results of this nature are the first ones for kinetic equations with integral (non-local) rough diffusion. 
 
This approach is reminiscent of the work presented in Lions' book \cite{MR153974} for parabolic equations:  see the remark on page 108 saying that this formula holds in all the examples treated there.   An explicit construction and proof of representation working in the local case at the highest level of generality is given by  M.~Egert and the first author in \cite{ae}. Using the language of distributions, the Schwartz kernels $\Gamma(t,x,v, s,y,w)$ of $\Gamma(t,s)$ can be thought of as the fundamental solution but, of course, it is not necessarily a function, unless more conditions are assumed. 
 
In the kinetic context, we use ideas from \cite{ae} in order to prove  that the representation \eqref{eq:superposition}  is true for all weak solutions when working at this level of generality. In particular, we explain the meaning of the integral, starting from the notion of Green operators. See Theorem~\ref{thm:representation0T} below and its variants in Section~\ref{sec:infiniteintervals}. We stress that we do not use approximations by operators with regular coefficients to do that.\\

In the second part of this article (Section~\ref{sec:localcase}), we turn to the case of local diffusions discussing the properties of our fundamental solution operators.  To simplify matters and explanations of ideas, we restrict to Fokker-Planck equations with pure second order operators;  the methods allow us to have lower-order terms with bounded coefficients but we shall not provide details. In this case, the elliptic part is given by  
 \begin{equation}
  \label{e:fk}
  \cA f(t,x,v)  = -\nabla_v \cdot ( \mathbf{A} \nabla_v f)(t,x,v) 
\end{equation}
where $\mathbf{A}=\mathbf{A}(t,x,v)$ is a bounded and measurable coercive matrix.  We import Davies's exponential argument \cite{MR1346221} to kinetic equations and show in our second main result, Theorem~\ref{thm:fundamental-bounds},  that this fundamental solution operator always enjoys $\L^2$ decay, measured in the distance induced by the underlying Galilean group law as it should be. Furthermore, we prove the conservation property (Theorem~\ref{thm:conservationproperty}) also at this level of generality.  
   
In these arguments, keeping in mind that kinetic models are naturally real-valued, $\cA$ could even be a complex-valued operator (and even a system). Of course, the motivation coming from nonlinear kinetic equations (such as the Landau equation) only involves real coefficients and real solutions. We note that handling complex coefficients opens the door to analytic perturbation results looking at the dependency of the solutions or the fundamental solution with respect to the coefficients of $\cA$ in $\L^\infty$ topology. 
   
 The next natural question is whether the fundamental solution operator has an integral representation with a kernel having pointwise kinetic Gaussian upper estimates as when $\cA=-\Delta_{v}$. There are known methods for real coefficients,  assuming \textit{a priori} the existence of the fundamental solution and not using the regularity of the coefficients quantitatively, see below for the literature review.  Nevertheless, our third main result, Theorem~\ref{thm:MoserimpliesGUB}, shows that the integral representation holds and that upper estimates for its kernel can be derived using, as a black box, scale-invariant $\L^2-\L^\infty$ Moser estimates for local weak solutions in the sense defined in Section~\ref{sec:weaksol}  of the equation and of the adjoint equation. This is obtained by adapting an argument of S. Hofmann and S. Kim  \cite{MR2091016} done in a parabolic context. 
 
 In fact, Moser's estimates are also necessary. These are, therefore, equivalent statements. In particular, combining this with the work of F.~Golse, C.~Mouhot, A.~F. Vasseur and the second author in \cite{MR3923847}, the conclusion is the existence  of a generalized  fundamental solution with the expected properties  when $\cA$ has real measurable coefficients 
 (Theorem~\ref{thm:conclusion}). We also show uniqueness, which tells us that our construction subsumes all other ones done under supplementary hypotheses, making therefore all various results and methods of the literature on fundamental solutions  fully available.\\

We finish this introduction by reviewing the main contributions on these topics. They are only concerned with real coefficients. In the local case, existing results on Moser $\L^2-\L^\infty$ estimates
can explicitly be found for the first time in a work of  A.~Pascucci and S.~Polidoro \cite{MR2068847} for a notion of weak solutions with $(\partial_t + v \cdot \nabla_x ) f \in \L^2_{\loc}$ in all variables, which has no reason to hold in the case of rough coefficients. Assuming symmetry of the coefficients (but this is not necessary), this restriction was lifted by F.~Golse, C.~Mouhot, A.~F. Vasseur and the second author in \cite[Lemma 11]{MR3923847} and  Moser's estimate was proved for weak solutions locally in $\L^\infty_{t}\L^2_{x,v} \cap \L^2_{t,x}\H^1_{v}$  using energy inequalities for sub-solutions. A Harnack inequality was also proved there (\cite[Theorem 3]{MR3923847}). See also the works of J.~Guerand, C.~Mouhot and the second author \cite{MR4653756,MR4453413} for further developments.

As far as Gaussian upper and lower pointwise estimates of the fundamental solution are concerned, they were first obtained in the local case assuming its \textit{a priori} existence and working with the class of weak solutions of \cite{MR2068847}. We mention three works along this direction.  A.~Pascucci and S.~Polidoro \cite{MR2000352}  proved upper estimates for a larger class of ultra-parabolic equations with H\"older continuous coefficients.  A.~Lanconelli and A.~Pascucci obtained upper estimates in \cite{MR3717350} by the Nash method for real measurable coefficients. Under the same assumptions as in \cite{MR3717350}, lower bounds are obtained by  A.~Lanconelli, A.~Pascucci and S. Polidoro in \cite{MR4181953}.

Recently, for larger classes of weak solutions \pascal{than in \cite{MR2068847}, but still smaller than ours},  F.~Anceschi and A.~Rebucci \cite{francesca2021fundamental} prove the existence of a fundamental solution in the sense of \cite{MR4181953} for real and symmetric measurable coefficients under a further technical integrability condition: their construction uses regularisation of the coefficients and bounds for fundamental solutions when coefficients are H\"older continuous;  they also derive Gaussian upper and lower bounds under a general existence assumption.

In the parabolic case with non-local real diffusions, the article of M.~Kassmann and M.~Weidner \cite{MR4530315} and the references therein yield the most recent update on works providing upper and lower pointwise estimates, using Davies' or Aronson's methods. This should provide insight for further developments in the kinetic context.

\section{Weak solutions to Kolmogorov-Fokker-Planck equations}
\label{sec:review}

For the comfort of the reader and also to prepare the grounds for the construction of fundamental solutions, we summarize the theory developed in \cite{AIN}. In particular, we recall the functional spaces we will work with, the assumption on the diffusion operator $\mathcal{A}$, the definition of weak solutions and the results concerning their existence in various settings.

\pascal{For $\beta\in \R$, we denote the homogeneous and inhomogeneous Sobolev norms on $\R^d$ in the $v$-variable by
\[
\|f\|_{\Hdot^\beta_{v}}=\|D_{v}^\beta f\|_{\L^2_{v}}, \quad  \|f\|_{\H^\beta_{v}}^2 = \|(I-\Delta_{v})^{\beta/2} f\|_{\L^2_{v}}^2, \]
where $D_{v}=(-\Delta_{v})^{1/2}$, \pascal{which amounts to multiplication by $|\xi|$ in Fourier variable.} For 
$\beta\in \R$,  the Sobolev space $\H^\beta_{v}$ is the space of tempered distributions with $\|f\|_{\H^\beta_{v}}<\infty$. For the homogeneous norm, the expression $D_{v}^\beta f$ is not defined for all tempered distributions and, as we do not want to work modulo polynomials, we have to proceed as follows. The image of $\L^2_{v}$ under  $D_{v}^{-\beta}$  is  a set of tempered distributions  if and only if $\beta<d/2$. For $\beta <d/2$, we define the homogeneous Sobolev space $\Hdot^\beta_{v}$ as the space of tempered distributions $f$ with  $f=D_{v}^{-\beta}g$ for some $g\in \L^2_{v}$, and we write $g=  D_{v}^{\beta} f$ and $\|f\|_{\Hdot^\beta_{v}}=\|g\|_{\L^2_v}$. It is a Hilbert space.  For $\beta\ge d/2$, given a tempered distribution $f$ with further assumptions, \textit{e.g.} $f\in \L^2_{v}$, $D_{v}^\beta f$ exists as a tempered distribution and one can consider its $\L^2_{v}$ norm.  Even if this does not lead to a proper definition of a Hilbert space, we shall use this remark later on.}  

For a given interval $I$ of $\R$, we assume that  $I\times \R^d\ni (t,x)\mapsto a_{t,x}$  is a  family of continuous sesquilinear  forms 
$$a_{t,x}\colon \H^{\beta}_{v} \times \H^\beta_{\vphantom{v} v} \to \IC$$
such that there exists $\Lambda <\infty, 0\le c^0<\infty$ such that for all  
$f,g \in \H^{\beta}_{v}$, the function
   $(t,x)\mapsto a_{t,x}(f,g)$ is measurable and satisfies 
\begin{equation}
  \label{e:ellip-upper}
  |a_{t,x}(f,g)|\le \Lambda (\|f\|_{\Hdot^\beta_{v}}+c^0\|f\|_{\L^2_{v}})( \|g\|_{\Hdot^\beta_{v}}+c^0\|g\|_{\L^2_{v}})
\end{equation}
 uniformly in $(t,x)\in I\times \R^d$.  We set
\begin{equation}
  \label{e:a-defi}
  \forall f,g \in \L^2_{t,x}\H^{\beta}_{\vphantom{t,x} v}, \qquad a(f,g)=\iint_{I\times \R^d} a_{t,x}(f,g)\, \dt\dx.
\end{equation}
For short, we have written   $a_{t,x}(f(t,x),g(t,x))$ as $a_{t,x}(f,g)$ in \eqref{e:a-defi}. From \eqref{e:ellip-upper}, we deduce that $a$ is a continuous  sesquilinear form on $\L^2_{t,x}\H^{\beta}_{\vphantom{t,x} v}$.
As $\L^2_{t,x}\H^{-\beta}_{\vphantom{t,x} v}$ identifies with the dual of $\L^2_{t,x}\H^{\beta}_{\vphantom{t,x} v}$ in the duality extending the  $\L^2_{t,x,v}$ inner product, we can set  $\cA: \L^2_{t,x}\H^{\beta}_{\vphantom{t,x} v} \to \L^2_{t,x}\H^{-\beta}_{\vphantom{t,x} v}$ the bounded operator associated to $a$ by 
$$ \angle{\cA f}{g}=  a(f,g), \quad f,g\in \L^2_{t,x}\H^{\beta}_{\vphantom{t,x} v}.$$
\pascal{If we assume coercivity in the sense that    there exists $\lambda>0, c_{0}\ge 0$ such that for all $(t,x)\in I\times \R^d$ and  all $f\in \H^\beta_{v}$,
\begin{align}
\label{eq:weakelliptic-bis}
 \Re a_{t,x}(f,f) \ge \lambda \|f\|_{\Hdot^\beta_{v}}^2 -c_{0}\|f\|^2_{\L^2_{v}} \ ,
 \end{align}
then we get from the Lax-Milgram lemma that   $\cA+c$ is an isomorphism for  $c>c_{0}$. }

\pascal{If $I$ is infinite,  $c^{0}=0$ in \eqref{e:ellip-upper} and  $c_{0}=0$  in  \eqref{eq:weakelliptic-bis}, then one can work with  $\cA$:   it  is extended to a bounded and invertible operator from $\L^2_{t,x}\Hdot^{\beta}_{\vphantom{t,x} v}$ to $\L^2_{t,x}\Hdot^{-\beta}_{\vphantom{t,x} v}$ if $\beta<d/2$. If $\beta\ge d/2$,  $\cA$ is extended to $\L^2_{t,x}\Hdot^{\beta}_{v} \cap \L^2_{\loc, t}\L^2_{x,v}$ into $\L^2_{t,x}\Hdot^{-\beta}_{\vphantom{t,x} v}$
still using the homogeneous norms in the upper and  lower bounds.} 

%
%

In the notation $ \L^2_{t,x}\H^{\beta}_{\vphantom{t,x} v}$, we do not indicate the time interval. We do it for the other mixed spaces of the same type. It will be clear from the context.

 For the next definition, we only need to have $\cA$  defined and bounded as above.
 
 \begin{defn}[Weak solutions to the kinetic Cauchy problem]
\label{defn:weaksol-cauchy} Let $0<T<\infty$ \pascal{and $\beta>0$.}
 Let $S \in \cD' ((0,T)\times \R^d\times \R^d)$  and $\psi \in \L^2_{x,v}$. Assume that the  distribution $S$ extends to a continuous linear functional on $\cD ([0,T)\times \R^{d}\times \R^d)$. A  distribution $f \in \cD' ((0,T)\times \R^{d}\times\R^d)$ is said to be a \emph{weak solution} \pascal{to the Cauchy problem on $[0,T]$}  \begin{align}
 \label{eq:CP}
 \begin{cases}
 (\partial_{t}+v\cdot\nabla_{x})f + \cA f= S,  & \\
  f(0)=\psi,  &
\end{cases}
\end{align}
 if $f\in  \L^2_{t,x}\H^{\beta}_{\vphantom{t,x} v}$ and for all $h\in \cD ([0,T)\times \R^{d}\times\R^d)$, 
\[
  -\iiint  f \, (\partial_{t}+v\cdot\nabla_{x}) \overline h \, \dt \dx \dv + a(f,h) = \angle{S}{h} + \iint_{\R^{2d}} \psi (x,v) \overline h (0,x,v) \dx \dv.
\]
\end{defn}  

The sources $S$ in our result have a clear extension, so we do not insist on the definition of this extension in full generality \pascal{and refer to the assumptions of our next result. For the definition when $T=\infty$, one may not want to  require $f\in  \L^2_{t,x}\H^{\beta}_{\vphantom{t,x} v} $ if $c^{0}=0$ in \eqref{e:ellip-upper} ({as there would not be such a solution})  and only ask for   $f\in  \L^2_{t,x}\Hdot^{\beta}_{\vphantom{t,x} v}$ when $\beta<d/2$ and $f\in \L^2_{t,x}\Hdot^{\beta}_{v} \cap \L^2_{\loc, t}\L^2_{x,v}$ 
   if $\beta\ge d/2$ (see our discussion  at the beginning of Section \ref{sec:review}) {so that} the variational formulation makes sense.} If we work with $I=\R$, then there is no initial value $\psi$ and only the equation
\begin{equation}
\label{eq:weaksol}
 (\partial_t + v \cdot \nabla_x ) f +\cA f = S \ \pascal{\mathrm{in} \  \cD' (\R\times \R^{d}\times\R^d).}
\end{equation} Working on $\R$ can also be thought of as a kinetic Cauchy problem on $[-\infty,\infty]$ with initial condition 0 at $-\infty$.

The following theorem summarizes the results in \cite{AIN}.

\begin{thm}[Construction of weak solutions to the kinetic Cauchy problem]
  \label{thm:CP0T}
Let $0<T<\infty$ \pascal{and $\beta>0$}.   Assume that  the sesquilinear form $a$ satisfies  \eqref{e:ellip-upper} and \eqref{eq:weakelliptic-bis}  on $(0,T)\times \R^d\times \R^d$.
 Let $S\in  \L^2_{t,x}\H^{-\beta}_{v} + \L^2_{t,v}\H_{\vphantom{t,v} x}^{-\frac{\beta}{2\beta+1}}
+ \L^1_{t}\L^2_{x,v}$ and $\psi\in \L^2_{x,v}$. There exists a unique weak solution  $f\in \L^2_{t,x}\H^{\beta}_{v}$ to 
 \begin{align}
 \label{eq:CP0T}
 \begin{cases}
 (\partial_{t}+v\cdot\nabla_{x})f + \cA f= S,  &
 \\
  f(0)=\psi &
\end{cases}
\end{align}
in the sense of Definition \ref{defn:weaksol-cauchy}.
Moreover, $f\in  \L^2_{t,v}\H_{\vphantom{t,v} x}^{\frac{\beta}{2\beta+1}}\cap \C([0,T]\, ; \L^2_{x,v} )$,  $f(t)$ converges to $\psi$ in  $ \L^2_{x,v}$ as $t\to 0$, and 
 \begin{align}
\nonumber
\sup_{t\in [0,T]} \|f(t)\|_{\L^2_{x,v}} +  \|  D_{v}^{\beta} f\|_{\L^2_{t,x,v}}&+  \|   D_{\vphantom{D_{v}} x}^{\frac{\beta}{2\beta+1}} f\|_{\L^2_{t,x,v}} +  \|   f\|_{\L^2_{t,x,v}}\\
&
 \label{eq:estimateweaksolution}\lesssim \|S \|_{\L^2_{t,x}\H^{-\beta}_{\vphantom{t,x} v} + \L^2_{t,v}\H_{\vphantom{t,x} x}^{-\frac{\beta}{2\beta+1}}+ \L^1_{t}\L^2_{x,v}} +\|\psi\|_{\vphantom{\L^2_{t,x}\H^{-\beta}_{\vphantom{t,x} v} + \L^2_{t,v}\H_{\vphantom{t,x} x}^{-\frac{\beta}{2\beta+1}}+ \L^1_{t}\L^2_{x,v}} \L^2_{x,v}}
\end{align}
for an implicit constant  depending on $d,\beta, \lambda, \Lambda, c_{0}, c^0, T$. Furthermore, $f$ satisfies the energy equality for all times $0 \le s < t \le T$
\begin{align} \label{e:energy}
\nonumber \|f(t)\|^2_{\L^2_{x,v}}  + 2\Re& \int_{s}^t\int_{\R^d} a_{\tau,x}({f},{f})\dx\dd \tau = 
\\ \|f(s)\|^2_{\L^2_{x,v}}   &+ 2\Re \int_{s}^t\bigg(\int_{\R^d} \angle {S_{1}}{f} \dx + \int_{\R^d} \angle {S_{2}}{f} \dv + \iint_{\R^{2d}}  {S_{3}}\overline f \dx \dv \bigg)\dd \tau
  \end{align} 
  for any decomposition $S=S_{1}+S_{2}+S_{3}$ with $S_{1}\in \L^2_{t,x}\H^{-\beta}_{v}$, $S_{2}\in \L^2_{t,v}\H_{\vphantom{t,v} x}^{-\frac{\beta}{2\beta+1}}$ and $S_{3}\in \L^1_{t}\L^2_{x,v}$.
  
%
\end{thm}

We recall that the theorem on $[0,T]$ proceeded from a similar one on  $[0,\infty)$  by restriction (by taking what we called the canonical extension of $\cA$, see the proof of Theorem 5.7 in \cite{AIN}) and that the study on $[0,\infty)$  proceeded by a modification of the argument on $\R$  to incorporate the initial condition and use of the causality principle. 
We describe their specific statements in several remarks.     

\begin{rem}[Case of $\R$ with inhomogeneous operators and spaces]
\label{rem:1}
  The equation  $(\partial_{t}+v\cdot\nabla_{x})f+ \cA f +cf =S$ with $c>c_{0}$ in \eqref{eq:weakelliptic-bis} \pascal{has a unique weak solution in $\L^2_{t,x}\H^{\beta}_{v}$. This solution belongs to 
$\L^2_{t,x}\H^{\beta}_{v}\cap \L^2_{t,v}\H_{\vphantom{t,v} x}^{\frac{\beta}{2\beta+1}}\cap \C_{0}(\R\, ; \L^2_{x,v} )$ (continuity with $\lim_{t\to \pm\infty}\|f(t)\|_{\L^2_{x,v}} =0$),  the same estimates as in \eqref{eq:estimateweaksolution} on $\R$ and {the} energy equality \eqref{e:energy}. In other words, this proves the invertibility of  $(\partial_{t}+v\cdot\nabla_{x})+ \cA  +c$ and that its inverse, that we denote by $\cK_{\cA+c}\, $, has the continuity properties given by the estimates on $f$ from the source $S$.}  
\end{rem}

\begin{rem}[Case of $[0,\infty)$ with inhomogeneous operators and spaces]
\label{rem:2}
The kinetic Cauchy problem on $[0,\infty)$ is posed for the operators $(\partial_{t}+v\cdot\nabla_{x}) + \cA +c$ with $c>c_{0}$ in \eqref{eq:weakelliptic-bis} and data $(\psi,S)$. \pascal{It has a unique weak solution $f$ in $\L^2_{t,x}\H^{\beta}_{v}$, {which} belongs to 
$\L^2_{t,x}\H^{\beta}_{v}\cap \L^2_{t,v}\H_{\vphantom{t,v} x}^{\frac{\beta}{2\beta+1}}\cap \C_{0}([0,\infty)\, ; \L^2_{x,v} )$ (continuity with $\lim_{t\to \infty}\|f(t)\|_{\L^2_{x,v}} =0$), {satisfies} the same estimates as in \eqref{eq:estimateweaksolution} on $(0,\infty)$ and {the} energy equality \eqref{e:energy} on $[0,\infty)$.}  In addition, when $\psi=0$, \pascal{by the causality principle} $f=\cK_{\cA+c\,} S_{0}$ where $S_{0}$ is the zero extension of $S$ and $\cK_{\cA+c}$ is defined in Remark~\ref{rem:1}.
\end{rem}

\begin{rem}[Case of $\R$ with homogeneous operators and spaces]
\label{rem:3}
  When  $c^{0}=c_{0}=0$ in  \eqref{e:ellip-upper} and \eqref{eq:weakelliptic-bis} on $\R\times \R^d$, one works with homogeneous \pascal{norms},
  and proves an invertibility of the homogeneous operator $(\partial_{t}+v\cdot\nabla_{x})+ \cA$. \pascal{Namely,} 
  given $S$ in the corresponding homogeneous spaces \pascal{(with negative indices)}, there exists a weak solution satisfying the homogeneous estimates corresponding to \eqref{eq:estimateweaksolution} on $\R$ (eliminate the $\L^2_{t,x,v}$ estimate on $f$ and use homogeneous \pascal{Sobolev norms} for the source terms), the energy equality on $\R$  and $\lim_{t\to \pm\infty}\|f(t)\|_{\L^2_{x,v}} =0$.  \pascal{This holds for  all $\beta>0$ since the solution belongs to $\L^2_{x,v}$ for all times.  Due to the definition of the homogeneous Sobolev space $\Hdot^{\beta}_{v}$}, this solution is unique in  $\L^2_{t,x}\Hdot^{\beta}_{v}$ if $\beta<d/2$    and in $\L^2_{t,x}\Hdot^{\beta}_{v} \cap \L^2_{\loc,t}\L^{2}_{x,v}$ if $\beta\ge d/2$.  In \pascal{passing}, we proved invertibility of $(\partial_{t}+v\cdot\nabla_{x})+ \cA $ \pascal{(all $\beta>0$)}. Its inverse, that we denote $\cK_{\cA}\,$, has the continuity properties given by the estimates on $f$ from those of $S$. 
\end{rem}

\begin{rem}[Case of $[0,\infty)$ with homogeneous operators and spaces]
\label{rem:4}
 When  $c^{0}=c_{0}=0$ in  \eqref{e:ellip-upper} and \eqref{eq:weakelliptic-bis} on $(0,\infty)\times \R^d$, one  works with homogeneous Sobolev
 \pascal{norms} for the Cauchy problem on $[0,\infty)$  for the homogeneous operator $(\partial_{t}+v\cdot\nabla_{x})+ \cA$.
Given $S$ in the corresponding homogeneous spaces \pascal{and initial data $\psi\in \L^2_{x,v}$)}, there exists a weak solution satisfying the homogeneous estimates corresponding to \eqref{eq:estimateweaksolution} on $(0,\infty)$, the energy equality on $[0,\infty)$  and $\lim_{t\to \infty}\|f(t)\|_{\L^2_{x,v}} =0$. This solution is unique in  $\L^2_{t,x}\Hdot^{\beta}_{v}$ if $\beta<d/2$   and in $\L^2_{t,x}\Hdot^{\beta}_{v} \cap \L^2_{\loc,t}\L^{2}_{x,v}$ if $\beta\ge d/2$.   Moreover, when $\psi=0$, the causality principle says that $f=\cK_{\cA}\,S_{0}$ where $S_{0}$ is the zero extension of $S$ and $\cK_{A}$ is as in Remark~\ref{rem:3}. 
\end{rem}

\begin{rem}[Homogeneous local case]
\label{rem:5}
When $\beta=1$ and  $d\ge 1$, one can work in  Remarks~\ref{rem:3} and \ref{rem:4} 
with $\Wdot^{1,2}_{v}=\{f\in \cD'(\R^d)\, ;\, \nabla f\in \L^2(\R^d)\}$   (replacing $\Hdot^1_{v}$), 
 when $c^{0}=c_{0}=0$ in  \eqref{e:ellip-upper} and \eqref{eq:weakelliptic-bis}.  The weak solution in  Remark~\ref{rem:3} is unique  in  $\L^2_{t,x}\Wdot^{1,2}_{v}$ \pascal{modulo a constant (which can be eliminated if we require $f(t) \in \L^2_{x,v}$ for some $t$)} and  the one in Remark~\ref{rem:4} is unique in  $\L^2_{t,x}\Wdot^{1,2}_{v}$.
\end{rem}

\begin{rem}[The backward Cauchy problem]
 Consider the adjoint problem in $[0,T] \times \R^d \times \R^d$,
\begin{align}
 \label{eq:CP0Tback}
 \begin{cases}
 -(\partial_{t}+v\cdot\nabla_{x})\tf + \cA^* \tf= \widetilde S,  &
 \\
  \tf(T)=\tpsi &
\end{cases}
\end{align}
where $\cA^*$ is the adjoint of $\cA$.
We refer to it as the backward kinetic Cauchy problem on $[0,T]$ and its weak solutions are defined similarly. The analogous existence and uniqueness theorems hold on $[0,T]$, half-infinite intervals $(-\infty,T]$ and real line $\R$. See Remarks~5.4 and 5.8  in \cite{AIN}. 
\end{rem}
 
  \section{Construction of the fundamental solution} 
  \label{sec:construction}

In this section, we put ourselves in the conditions of Theorem~\ref{thm:CP0T} on a finite interval $[0,T]$. The situation on infinite intervals will be described in Section~\ref{sec:infiniteintervals}.

 Before we come to definitions, statements and proofs, we notice that one can of course shift the initial time and work on $[s,T]$ with $0\le s<T$. The same remark applies to the backward kinetic Cauchy problems for $-(\partial_{t}+v\cdot \nabla_{x})+\cA^*$ on intervals $[0,t]$ with $0<t\le T$.  This leads us to the following definition.  

 \subsection{Green operators}
 
 This terminology comes from Lions' book  \cite{MR153974}. 
 \begin{defn}[Green operators] 
 \label{defn:Greenoperators} Assume 
 $s,t \in [0,T]$ with $s\le t$, and $\psi,\tpsi \in \L^2_{x,v}$.
 \begin{enumerate}
 \item Define $\Ga(t,s)\psi$ as the value at time $t$  in $\L^2_{x,v}$ of the  weak solution $f$  to the kinetic Cauchy problem  
 $(\partial_{t}+v\cdot\nabla_{x})f + \cA f =0 $ with initial value $f(s)=\psi$. 
 \item Define $\tGa(s,t)\tpsi$ as the value at time $s$  in $\L^2_{x,v}$ of the   weak solution $\tilde f$ to the backward kinetic Cauchy problem
 $-(\partial_{t}+v\cdot\nabla_{x})\tf +\cA^*\tf= 0$ with final value $\tf(t)=\tpsi$. 
 \end{enumerate}
 The  families  $G=(\Ga(t,s))_{0\le s\le t\le T}$ and $\tGa=(\tGa(s,t))_{T\ge t\ge s\ge 0}$ are called the \emph{Green operator families} for the kinetic operator $(\partial_{t}+v\cdot \nabla_{x})+\cA $ and the (adjoint) backward kinetic operator $-(\partial_{s}+v\cdot \nabla_{x})+\cA^*$, respectively. We set $\Ga(t,s)=0$ and $\tGa(s,t)=0$ when $s>t$.  
 \end{defn}
 \begin{prop}[Properties of Green operators]
 \label{prop:propGreenop}
 The Green operators $G(t,s)$ and $\tilde G(s,t)$ satisfy the following properties. 
   \begin{enumerate}
  \item $\Ga(t,s)$ and $\tGa(s,t)$ are adjoint {uniformly bounded operators} on $\L^2_{x,v}$ with operator norm bounds $ e^{c_{0}(t-s)}$.
 \item The Chapman-Kolmogorov relations 
 $\Ga(t,s)=\Ga(t,r)\Ga(r,s)$ and $\tGa(s,t)= \tGa(s,r)\tGa(r,t)$ hold for  $s<r<t$. 
 \item $[s,T]\ni t\mapsto \Ga(t,s)$   and $[0,t] \ni s \mapsto \tGa(s,t)$ are strongly continuous on $\L^2_{x,v}$. 
 \end{enumerate}
 \end{prop}
 \begin{proof} To prove (i), 
  let $\psi,\tpsi\in \L^2_{x,v}$ and set  for $s\le \tau \le t$, $f(\tau)=\Ga(\tau,s)\psi$  and $\tilde{f}(\tau)=\tGa(\tau,t)\tpsi$. First, {from}
  Theorem~\ref{thm:CP0T},  {$f$ satisfies the energy equality on $[s,T]$, and by taking the derivative,
   and for almost every $t>s$, 
\begin{equation}
\label{eq:Gronwall}
 \frac{ \dd}{\dt} \|f(t)\|_{\L^2_{x,v}}^2 = -2\Re  \int_{\R^d} a_{t,x}(f,f)\dx \le 2 c_{0}\|f(t)\|_{\L^2_{x,v}}^2 
\end{equation}
  by \eqref{eq:weakelliptic-bis}. By Gronswall's lemma, 
  $ \|f(t)\|_{\L^2_{x,v}}^2 \le e^{2c_{0}(t-s)}\|\psi\|_{\L^2_{x,v}}^2$ and the boundedness 
  property follows}. The same argument applies to $\tilde f$ by reversing the sense of time. Let us now show the adjoint property. Assume $s<t$. 
 By polarisation of the energy equality for both $f$ and $\tf$,  see \cite[Theorem 3.8]{AIN},  $\tau\mapsto \angle{f(\tau)}{\tilde f(\tau)}$ (this is the $\L^2_{x,v}$ inner product) is absolutely continuous on the interval $[s,t]$  and
 \[
   \angle{f(t)}{\tilde f(t)}- \angle{f(s)}{\tilde f(s)}=  \int_{s}^t \int_{\R^d} -a_{\tau,x}(f,\tilde f) + a_{\tau,x}(f,\tilde f)\dx \dd\tau=0.
 \]
  As the left hand side equals $\angle{\Ga(t,s)\psi}{\tpsi}- \angle{\psi}{\tGa(s,t)\tpsi}$, this  proves the adjoint relation.
 
  To prove  (ii), we let  $\psi\in \L^2_{x,v}$, and set $f(\tau)=\Ga(\tau,s)\psi$ and $\tilde f(\tau)=\Ga(\tau,r)\Ga(r,s)\psi$ for $\tau\ge r$. We have $\tilde f(r)=f(r)$
  and using again the absolute continuity equalities, we see that $f$  solves the same kinetic Cauchy problem on $(r,T)\times \R^d\times \R^d$ as $\tilde f$, so by uniqueness, $f(t)$ and $\tilde f(t)$ agree on $[r,T]$.  
 
  The third assertion is an immediate consequence of the definition.
  \end{proof}

 \subsection{Fundamental solution}
 
 We define the fundamental solution as representing the inverse of 
 $(\partial_{t}+v\cdot\nabla_{x}) + \cA$ with zero initial condition. 

 We shall use the tensor product notation $\v \otimes \psi$ for products of functions $\v$ of the  $t$ variable and functions $\psi$ of the $(x,v)$ variables, that is $(\v \otimes \psi)(t,x,v)=\v(t)\psi(x,v)$. By a tensor product of test functions $\v \otimes \psi$, we mean  $\v\in \cD(0,T)$ and $\psi\in \cD(\R^{2d})$.   

 \begin{defn}[Definition of a fundamental solution]
 \label{def:FSOF}
 The fundamental solution to $(\partial_{t}+v\cdot\nabla_{x}) + \cA$ is a family of 
 operators $\Gamma=(\Gamma(t,s))_{ (s, t)\in [0,T]^2}$, called fundamental solution operators, with the following properties.
 \begin{enumerate}
 \item They are uniformly bounded operators on $\L^2_{x,v}$.
  \item $\Gamma(t,s)=0$  for almost every $(s,t)$ with $s>t$.
 \item  For any $\psi,\tpsi \in \cD(\R^{2d})$, $(s,t)\mapsto \angle {\Gamma(t,s)\psi}{\tpsi}$ is a locally integrable function:  
  \[\iint_{s\in K, \, t\in \tilde K} |\angle {\Gamma(t,s)\psi}{\tpsi}|\, \dd s\dt <\infty
   \]
   for compact subsets $K, \tilde K$ in $(0,T)$. 
   \item For any tensor product of test functions $S=\phi\otimes \psi$, the weak solution $f$ to  the kinetic Cauchy problem 
   $(\partial_{t}+v\cdot\nabla_{x})f + \cA f=S$ on $(0,T)$  with $f(0)=0$, satisfies
   \begin{align}
   \label{eq:represensation}
  \angle{f(t)}{\tpsi}=\int_{(0,t)} \phi(s)\angle {\Gamma(t,s)\psi}{\tpsi}\, \dd s
 \end{align}
  for all $\tpsi\in \cD(\R^{2d})$ and almost every $t\in (0,T]$. 
  \end{enumerate}
  
  One defines  the fundamental solution $\widetilde \Gamma=(\widetilde \Gamma(s,t))_{ (s,t) \in [0,T]^2}$  to the backward operator   $-(\partial_{s}+v\cdot \nabla_{x})+\cA^*$ analogously  and (ii) is replaced by  $\widetilde \Gamma(s,t)=0$ if $s>t$.
  \end{defn}

 \begin{lem}[Uniqueness of fundamental solutions]
 \label{lem:FSuniqueness}
 There is at most one fundamental solution in the sense of Definition~\ref{def:FSOF}.
  \end{lem}
 
 \begin{proof} 
Suppose  $\Gamma, \Gamma'$ are two fundamental solutions.  \pascal{Applying the items (iii), (ii) and (iv)} with source terms $S$ consisting of  tensor products of test functions $\v \otimes \psi$, and testing the obtained weak solution   against a tensor $\tv\otimes \tpsi$ of the same type, we have from \eqref{eq:represensation} using Fubini's theorem (allowed by (iii)),   
 \[ \iint_{0< s \le t< T} \v(s)\angle {\Gamma(t,s)\psi}{\tpsi}\overline{\tv}(t)\, \dd s\dt  =  \iint_{0 < s \le t<T}  \v(s)\angle{\Gamma'(t,s)\psi}{\tpsi}\overline{\tv}(t)\, \dd s\dt. 
 \]
 Letting $\v$, $\tv$  approximate Dirac masses at some $s$ and $t$ respectively,  we obtain from Lebesgue's differentiation theorem (justified  by the local integrability condition (iii)) 
 \[\angle {\Gamma(t,s)\psi}{\tpsi}=\angle {\Gamma'(t,s)\psi}{\tpsi}
 \]
 for almost every $(s,t)$ with $0<s < t<T$.  At this stage, the almost everywhere equality can depend on $\psi,\tpsi$. Applying this for test functions $\psi,\tpsi$ describing a countable total set in $\L^2_{x,v}$, and using that the two operators $\Gamma(t,s), \Gamma'(t,s)$ are bounded on $\L^2_{x,v}$ by (i), we easily conclude they agree almost everywhere.
  \end{proof}

  \subsection{Green operators  and fundamental solution operators are the same}

 Having defined two families, we now show they agree. This is where it is useful to have proved the well-posedness of the kinetic Cauchy problem for $\L^1_{t}\L^2_{x,v}$ sources.

 \begin{thm} 
 \label{thm:exunfundsol} The family of Green operators is the fundamental solution (up to almost everywhere equality).   
  \end{thm}

 \begin{proof} We partly borrow  ideas from \cite{ae}. Given that there is uniqueness of a fundamental solution, up to equality almost everywhere, by Lemma~\ref{lem:FSuniqueness}, it suffices to prove that the Green operators satisfy the requirements in Definition~\ref{def:FSOF}.
  
  That Green operators are uniformly bounded follows from Proposition~\ref{prop:propGreenop}. Item (ii) is in the definition of $G(t,s)$.  We set $C=\sup \|\Ga(t,s)\| \le e^{c_{0}T}$.

 Next, we elucidate measurability issues. Remark that for $\psi,\tpsi \in \L^2_{x,v}$, since $(s,t)\mapsto \angle {\Ga(t,s)\psi}{\tpsi}= \angle {\psi}{\tGa(s,t)\tpsi}$ is separately continuous on $\{(s,t)\in [0,T]^2\, ; \,  s\le t\}$, it is a (Borel) measurable function on this set. Furthermore,  
 \begin{equation}
 \label{eq:|kernel|}
 \iint_{0<s \le t<T} \big| \v(s)\angle {\Ga(t,s)\psi}{\tpsi}\overline\tv(t) \big|\ds \dt \le C  \|\v\|_{\L^1(0,T)}\|\tv\|_{\L^1(0,T)}\|\psi\|_{\L^2_{x,v}}\|\tpsi\|_{\L^2_{x,v}}. 
 \end{equation}
 The third item in Definition~\ref{def:FSOF} follows.  
 It remains to prove the last one.

 Let $\phi\in \L^1(0,T)$, $\psi, \tpsi \in \L^2_{x,v}$. Let $f$ be the weak solution to the kinetic Cauchy problem on $(0,T)$ with source $\v\otimes \psi\in \L^1_{t}\L^2_{x,v}$ and initial value 0, and   $\tilde f(s)=\tGa(s,t)\tpsi$ for $0\le s\le t$ with fixed $t$. Then $\tilde f(t)=\tpsi$ and $f(0)=0$. Thus, using again the absolute continuity of $s\mapsto \angle{f(s)}{\tilde f(s)}$ on $[0, t]$, since
 $ \angle{f(t)}{\tpsi}= \angle{f(t)}{\tilde f(t)}- \angle{f(0)}{\tilde f(0)}$, we have
  $$\angle{f(t)}{\tpsi}=  \int_{0}^t \iint_{{\R^{2d}}} \v\otimes\psi \ \overline{\tilde f} \dv \dx \ds.
  $$
   Using  the definition of $\tilde f$ and $\tGa(s,t)=\Ga(t,s)^*$, we obtain
  $$
   \angle{f(t)}{\tpsi}=  \int_{0}^t \v(s)\angle {\psi}{\tGa(s,t)\tpsi}\ds= \int_{0}^t \v(s)\angle {\Ga(t,s)\psi}{\tpsi}\ds$$
     as desired. 
 \end{proof}

 \subsection{Representation of weak solutions to the kinetic Cauchy problems}

Having identified Green operators to fundamental solution operators, the latter inherit the properties of the former as stated in Proposition~\ref{prop:propGreenop}.  We use from now on the more traditional notation $\Gamma(t,s)$.  We may state a complete representation theorem for {all} weak solutions,  with specified convergence issues. If not mentioned, we mean that the time variable belongs to $(0,T)$. 

 \begin{thm}[Representation of weak solutions to the kinetic Cauchy problem] \label{thm:representation0T}
 Under the assumptions of  Theorem~\ref{thm:CP0T} on $(0,T)$ for $S$ and $\psi$, 
  the corresponding weak solution $f$  is represented  for any $t\in [0,T]$ in $\L^2_{x,v}$  by
  \begin{align}
 \label{eq:representationweaksolCP0T}
 f(t)=\Gamma(t,0)\psi+ \int_{0}^t \Gamma(t,s) S(s)\ds
 \end{align}
 where the integral is meant in the following sense.

 If $S\in \L^1_{t}\L^2_{x,v}$, then the integral is defined in the strong (Bochner) sense. 

 If $S= D_{v}^\beta\, F$ with $F \in  \L^2_{t,x,v}$, the integral is defined weakly in $\L^2_{x,v}$ by $($with $\L^2_{x,v}$ inner products$)$
 \begin{align}
 \label{eq:representationweaksol1}
 \angle{f(t)}{\tpsi}=\int_{0}^t \angle{F(s)}{D_{v}^\beta\, \widetilde\Gamma(s,t)\tpsi}\ds, \quad \tpsi \in \L^2_{x,v}.
 \end{align}

 If $S= D_{x}^{\frac{\beta}{2\beta+1}} F$ with $F \in  \L^2_{t,x,v}$, the integral is defined weakly in $\L^2_{x,v}$ by $($with $\L^2_{x,v}$ inner products$)$
 \begin{align}
 \label{eq:representationweaksol2}
 \angle{f(t)}{\tpsi}=\int_{0}^t \angle{F(s)}{D_{x}^{\frac{\beta}{2\beta+1}}\, \widetilde\Gamma(s,t)\tpsi}\ds, \quad \tpsi \in \L^2_{x,v}.
 \end{align}

 The analogous  representation  holds for  weak solutions to the backward kinetic Cauchy problem  for 
 $-(\partial_{t}+v\cdot \nabla_{x})+\cA^*$ with $\widetilde \Gamma(s,t)$. 
 \end{thm}

For simplicity of the presentation, we forget some parentheses in our notation here and subsequently. For example, $\Gamma(s,t)\tpsi $ should be 
 $\Gamma(s,t)(\tpsi)$,  $D_{v}^\beta\, \widetilde\Gamma(s,t)\tpsi$ should be  $D_{v}^\beta\, (\widetilde\Gamma(s,t)(\tpsi))$, that is $D_{v}^\beta$ applied to the function  $(x,v)\mapsto (\widetilde\Gamma(s,t)(\tpsi))(x,v)$. We do the same for $D_{x}^{\frac{\beta}{2\beta+1}}\, \widetilde\Gamma(s,t)\tpsi$.

 \begin{lem}[Integral estimates from the Green operators]
 \label{prop:SFE} For all $\psi, \tpsi \in \L^2_{x,v}$, $s,t\in [0,T]$, we have 
 \begin{equation}
 \int_{s}^T \|D_{v}^\beta\, \Gamma(t,s)\psi\|^2_{\L^2_{x,v}} \dt +  \int_{s}^T \|D_{\vphantom{v} x}^{\frac{\beta}{2\beta+1}}\, \Gamma(t,s)\psi\|^2_{\L^2_{x,v}}\dt \lesssim\|\psi\|^2_{\L^2_{x,v}} .\end{equation}
 \begin{equation}
 \int_{0}^t \|D_{v}^\beta\, \widetilde\Gamma(s,t)\tpsi\|^2_{\L^2_{x,v}}\ds +  \int_{0}^t \|D_{x}^{\frac{\beta}{2\beta+1}}\, \widetilde\Gamma(s,t)\tpsi\|^2_{\L^2_{x,v}}\ds \lesssim\|\tpsi\|^2_{\L^2_{x,v}} .\end{equation}
 The implicit constants depend on $d,\beta, \lambda,\Lambda, c_{0}, c^0, T$.
 \end{lem}
\begin{proof}
 This is a reformulation of the estimates for weak solutions to the forward and backward kinetic Cauchy problems with zero source in terms of integral estimates for the Green operators $\Ga(t,s)$, which are now identified to $\Gamma(t,s)$.
\end{proof}

 \begin{proof}
  [Proof of Theorem~\ref{thm:representation0T}] If $S=0$, then we know that $f(t)=\Ga(t,0)\psi=\Gamma(t,0)\psi$. We next assume that the initial value is 0. For a tensor product of test functions $S=\v\otimes \psi$, we have established  \eqref{eq:represensation} for all $\tpsi\in \cD(\R^{2d})$ and almost every $t\in (0,T]$. 
 We extend its validity. Namely, \begin{align}
  \label{eq:represntationallpsiallt}
  \angle{f(t)}{\tpsi}=\int_{0}^t \phi(s)\angle {\Gamma(t,s)\psi}{\tpsi}\, \dd s
 \end{align} holds for all $\tpsi\in \L^2_{x,v}$ and all $t\in [0,T]$.  
  First, the integrand in the right-hand side is well-defined for $\tpsi\in \L^2_{x,v}$ so we may take $\tpsi\in \L^2_{x,v}$ as $f(t)\in \L^2_{x,v}$. Secondly, the integrand is continuous on $[0,t]$ as a function of $s$. Thirdly,   it is not hard to see that as a function of $t$, the integral in the right-hand side with $\tpsi\in \L^2_{x,v}$   is continuous: the equality holds for all $t$ and not just almost everywhere. 
  
  Having proved \eqref{eq:represntationallpsiallt}, it remains to fix $t$ and argue by density arguments.    When $S\in \L^1_{t}\L^2_{x,v}$, we proceed as follows.  
 We see that  $s\mapsto |\angle{\Gamma(t,s)S(s)}\tpsi| = |\angle{S(s)}{\widetilde \Gamma(s,t)\tpsi}| $ is integrable on $(0,t)$ for all $\tpsi \in \L^2_{x,v}$
 and this shows that the integral in \eqref{eq:representationweaksolCP0T} weakly converges in $\L^2_{x,v}$. As $\L^2_{x,v}$ is a separable Hilbert space,  it is classical that this is equivalent to strong measurability of $s\mapsto {\Gamma(t,s)S(s)}$ on $(0,t)$ in $\L^2_{x,v}$ (\cite[Chapter 1]{MR3617205}), and thus the integral is in the sense of Bochner on $(0,t)$ as an $\L^2_{x,v}$ function. We now argue by density since the above tensors are total in $\L^2_{x,v}$. 

 The span of tensor products of test functions $\v\otimes\psi$ also is a dense subspace in the other inhomogeneous source spaces in the statement of the kinetic Cauchy problem.    We conclude again with a density argument using the estimates in Lemma~\ref{prop:SFE}. 
 \end{proof}

 Having proved the representation, we can reformulate the estimates for weak solutions to the forward and backward kinetic Cauchy problems with zero initial data, in terms of the fundamental solution operators, completing the uniform boundedness on $\L^2_{x,v}$ and the ones in Lemma~\ref{prop:SFE}. 
 \begin{cor}[Integral estimates for the fundamental solution operators]
 \label{prop:SFE2}   For all $F$  or $\widetilde F $ for which the right-hand side is finite in each estimate, we have, with $D_{1},D_{2} \in \{D_{v}^\beta, D_{\vphantom{1} x}^{\frac{\beta}{2\beta+1}}, I\} $,
  \begin{equation}
 \label{eq:1}
 \int_{0}^T \bigg\|\int_{0}^t D_{1} \Gamma(t,s)D_{2} F(s)\, \ds\bigg\|^2_{\L^2_{x,v}} \dt  \lesssim \int_{0}^T\|F(s)\|^2_{\L^2_{x,v}} \ds .\end{equation}
 \begin{equation}
 \label{eq:2}
 \sup_{t\in [0,T]} \bigg\|\int_{0}^t  \Gamma(t,s)D_{2} F(s)\, \ds\bigg\|^2_{\L^2_{x,v}} 
 \lesssim \int_{0}^T\|F(s)\|^2_{\L^2_{x,v}} \ds.
 \end{equation}
 \begin{equation}
 \label{eq:3}
 \sup_{t\in [0,T]} \bigg\|\int_{0}^t  \Gamma(t,s) F(s)\, \ds\bigg\|_{\L^2_{x,v}} 
 \lesssim \int_{0}^T\|F(s)\|_{\L^2_{x,v}} \ds.
 \end{equation}
 \begin{equation}
 \label{eq:4}
 \int_{0}^T \bigg\|\int_{s}^T D_{1} \widetilde \Gamma(s,t)D_{2} \widetilde F(t)\, \dt\bigg\|^2_{\L^2_{x,v}} \ds  \lesssim \int_{0}^T\|\widetilde F(t)\|^2_{\L^2_{x,v}} \dt .\end{equation}
 \begin{equation}
 \label{eq:5}
 \sup_{s\in [0,T]} \bigg\|\int_{t}^T  \widetilde \Gamma(s,t)D_{2} \widetilde F(t)\, \dt\bigg\|^2_{\L^2_{x,v}} 
 \lesssim \int_{0}^T\|\widetilde F(t)\|^2_{\L^2_{x,v}} \dt.
 \end{equation}
 \begin{equation}
 \label{eq:6}
 \sup_{s\in [0,T]} \bigg\|\int_{s}^T  \widetilde \Gamma(s,t) \widetilde F(t)\, \dt\bigg\|_{\L^2_{x,v}} 
 \lesssim \int_{0}^T\|\widetilde F(t)\|_{\L^2_{x,v}} \dt.
 \end{equation}
 The implicit constants depend on $d,\beta, \lambda,\Lambda, c_{0},c^0, T$. The integrals inside the norms in the left hand sides are all weakly defined by density, and the ones in \eqref{eq:2} and   \eqref{eq:5} with $D_{2}=I$, in  \eqref{eq:3}  and  \eqref{eq:6} are also strongly defined. 
 \end{cor}

Again, $D_{1} \Gamma(t,s)D_{2} F(s)$ should be thought as $D_{1} \big(\Gamma(t,s)(D_{2} F(s))\big)$, etc. 

 \subsection{The case of infinite intervals and/or homogeneous spaces}
 \label{sec:infiniteintervals}

 We describe the results in the situations of each remark after Theorem~\ref{thm:CP0T}, without proof. Note that, in some arguments with homogeneous spaces, one needs to replace $\psi, \tpsi \in \cD(\R^{2d})$ by $\psi, \tpsi$ in appropriate subsets of  $\cS(\R^{2d})$, see \cite{AIN}.

 \
 
 \paragraph{1) Case of $\R$ with inhomogeneous operators and  spaces}

 \

  We want to represent  $\cK_{\cA+c}\,$,  the inverse of  $(\partial_{t}+v\cdot \nabla_{x})+\cA+c$, $c>c_{0}$, under the assumptions \eqref{e:ellip-upper} and \eqref{eq:weakelliptic-bis} \pascal{as defined in Remark \ref{rem:1}}. In this case,   there is a unique fundamental solution $\Gamma$ (which we could denote by $\Gamma_{c}$) defined over $-\infty< s\le t<\infty$ and  $\Gamma(t,s)$ have   operator norms on $\L^2_{x,v}$  not exceeding $e^{-(c-c_{0})(t-s)}$.  This exponential decay can be seen from the use of Gronwall's inequality in \eqref{eq:Gronwall}. 

 \begin{prop}[Representation of weak solutions on $\R$ for inhomogeneous spaces] \label{thm:representationRinhom}
 For $S$   in the inhomogeneous spaces corresponding to those in Theorem~\ref{thm:CP0T}, $\cK_{\cA+c\,} S$  is represented  for any $t\in \R$ in $\L^2_{x,v}$ by
 \begin{align}
 \label{eq:representationweaksolRinhom}
 (\cK_{\cA+c\,} S)(t)=\int_{-\infty}^t \Gamma(t,s) S(s)\ds.
 \end{align}
The convergence for the integral is as in Theorem~\ref{thm:representation0T}, and there is strong convergence also when $S\in \L^2_{t,x,v}$ because of exponential decay. 
 The estimates of Lemma~\ref{prop:SFE} and Corollary~\ref{prop:SFE2} hold for $(-\infty,\infty)$ replacing $(0,T)$, with $D_{1},D_{2} \in \{D_{v}^\beta, D_{\vphantom{1} x}^{\frac{\beta}{2\beta+1}}, I\} $,  and the implicit constants depend on $d,\beta, \lambda,\Lambda, c-c_{0}, c^0$.
 \end{prop}

 \

 \paragraph{2) Case of $[0,\infty)$ with inhomogeneous operators and  spaces} 

 \

 We want to represent the weak solutions built in Theorem~\ref{thm:CP0T} replacing $\cA$ by $\cA+c$, $c>c_{0}$, in \eqref{eq:CP0T} \pascal{as in Remark \ref{rem:2}}. In this case, there is a  fundamental solution  $\Gamma$ defined over $0\le s\le t<\infty$ and  $\Gamma(t,s)$ have operator norms on $\L^2_{x,v}$ not exceeding $e^{-(c-c_{0})(t-s)}$.

 \begin{prop}[Representation of weak solutions on $[0,\infty)$ for inhomogeneous spaces] \label{thm:representation0inftyinhom}
  For $S$ as  in Theorem~\ref{thm:CP0T} with $T=\infty$,  the  weak solution $f$ to the kinetic Cauchy problem
 with $\cA$ replaced by $\cA+c$, is represented  for any $t\in [0,\infty)$ in $\L^2_{x,v}$ by
 \begin{align}
 \label{eq:representationweaksol0inftyinhom}
 f(t)=\Gamma(t,0)\psi+ \int_{0}^t \Gamma(t,s) S(s)\ds.
 \end{align}
 The convergence for the integral is as in Theorem~\ref{thm:representation0T} and strong convergence holds also when $S\in \L^2_{t,x,v}$. Moreover,  the integral agrees with $(\cK_{\cA+c\, }S_{0})(t)$ where $S_{0}$ is the zero extension of $S$ (and $\cA$ is extended in the canonical fashion). The estimates of Lemma~\ref{prop:SFE} and Corollary~\ref{prop:SFE2} hold for $(0,\infty)$ replacing $(0,T)$ with $D_{1},D_{2} \in \{D_{v}^\beta, D_{\vphantom{1} x}^{\frac{\beta}{2\beta+1}}, I\} $, and the implicit constants depend on $d,\beta, \lambda,\Lambda, c-c_{0}, c^0$.
 \end{prop}

\

 \paragraph{3) Case of $\R$ with homogeneous spaces} 

 \

 We want to represent $\cK_{A}$, the inverse of  $(\partial_{t}+v\cdot \nabla_{x})+\cA$ under the assumptions of Theorem~\ref{thm:CP0T} on $a$ with $c^{0}=c_{0}=0$, see Remark~\ref{rem:3} and Remark~\ref{rem:5}. In this case, there is a  unique fundamental solution  $\Gamma$ defined over $-\infty< s\le t<\infty$ and  $\Gamma(t,s)$ are contractive operators on $\L^2_{x,v}$. 

 \begin{prop}[Representation of weak solutions on $\R$ for homogeneous spaces] \label{thm:representationRhom} Assume   $c_{0}=c^0=0$.  
  For $S$ that belongs to the homogeneous spaces corresponding to those in Theorem \ref{thm:CP0T}, $\cK_{\cA}S$, that is the  weak solution $f$  to \eqref{eq:weaksol} in this result, is represented  for any $t\in \R$ in $\L^2_{x,v}$ by
 \begin{align}
 \label{eq:representationweaksolRhom}
 (\cK_{\cA}S)(t)=\int_{-\infty}^t \Gamma(t,s) S(s)\ds.
 \end{align}
 The convergence for the integral is as in  Theorem~\ref{thm:representation0T}.  The estimates of Lemma~\ref{prop:SFE} and Corollary~\ref{prop:SFE2} hold for $(-\infty,\infty)$ replacing $(0,T)$ with $D_{1},D_{2} \in \{D_{v}^\beta, D_{\vphantom{1} x}^{\frac{\beta}{2\beta+1}}\} $ and the implicit constants depend on $d,\beta, \lambda,\Lambda$.
 \end{prop}

 \
  \paragraph{4) Case of $[0,\infty)$ with homogeneous spaces} 

 \

 We want to represent the weak solutions to the kinetic Cauchy problem \eqref{eq:CP0T} on $[0,\infty)$ under the assumptions of Theorem~\ref{thm:CP0T} on $a$ with  $c^{0}=c_{0}=0$,  see Remark~\ref{rem:4} and Remark~\ref{rem:5}. In this case, there is a fundamental solution $\Gamma$ defined over $0\le s\le t<\infty$  and  $\Gamma(t,s)$ are contractive operators on $\L^2_{x,v}$. 

 \begin{prop}[Representation of weak solutions on $(0,\infty)$ for homogeneous spaces] \label{thm:representation0inftyhom} Assume   $c_{0}=c^0=0$.
  For $S$ that belongs to the homogeneous spaces corresponding to those in Theorem \ref{thm:CP0T},  the  weak solution $f$ to the kinetic Cauchy problem
\eqref{eq:CP0T} in this result is represented  for any $t\in [0,\infty)$ in $\L^2_{x,v}$ by
 \begin{align}
 \label{eq:representationweaksol0inftyhom}
 f(t)=\Gamma(t,0)\psi+ \int_{0}^t \Gamma(t,s) S(s)\ds.
 \end{align}
 The convergence for the integral is as in Theorem~\ref{thm:representation0T}. Moreover,  the integral agrees with $(\cK_{\cA}S_{0})(t)$ where $S_{0}$ is the zero extension of $S$ (and $\cA$ is extended in the canonical fashion). The estimates of Lemma~\ref{prop:SFE} and Corollary~\ref{prop:SFE2} hold for $(0,\infty)$ replacing  $(0,T)$ with $D_{1},D_{2} \in \{D_{v}^\beta, D_{\vphantom{1} x}^{\frac{\beta}{2\beta+1}}\} $, and the implicit constants depend on $d,\beta, \lambda,\Lambda$.

 \end{prop}

 \

\section{The case of local diffusion operators}
\label{sec:localcase}

In this section,  we assume that  the  $n\times n$ matrix $ \mathbf{A}(t,x,v)$ has  measurable  coefficients, bounded and elliptic in the sense of a G\aa rding inequality: \begin{equation}
  \label{e:ellip-lower}
\Re a_{t,x}(f,f):= \Re \int_{\R^d}\angle{\mathbf{A}(t,x,v)\nabla_{v} f}{\nabla_{v}f}\, \dv \ge \lambda \|\nabla_{v}f\|_{\L^2_{v}}^2
\end{equation} 
uniformly in $(t,x)$ and in $f\in \Wdot^{1,2}_{v}$.  Without loss of generality, we can assume  $I=\R$. The operator $\cA$ is given by \eqref{e:fk}. The sesquilinear form $a$ given by \eqref{e:a-defi}  on the space $\L^2_{t,x}\Wdot^{1,2}_{v}\times \L^2_{t,x}\Wdot^{1,2}_{v}$ satisfies  \eqref{e:ellip-upper} and \eqref{eq:weakelliptic-bis}   with $c^{0}=c_{0}=0$.

As evidenced in Remark~\ref{rem:5}, we can work directly with the fundamental solution operators $\Gamma(t,s)$ defined for all $-\infty<s\le t<\infty$.

\subsection{$\L^2$ decay}

At this level of generality, we always have $\L^2_{x,v}$ decay. 
\begin{thm}[$\L^2$ decay for the fundamental solution operator] \label{thm:fundamental-bounds}
  Assume the homogeneous assumptions of  Proposition~\ref{thm:representationRhom} with $\cA=-\partial_{v_{i}}a_{i,j}(t,x,v) \partial_{v_{j}}$. 
For all closed sets $E,F$ in $\R^{2d}$, $s<t$ and $\psi\in \L^2_{x,v}$, supported in $F$, 
\begin{equation}
\label{eq:L^2decay}
\|\Gamma(t,s) \psi\|_{\L^2_{x,v}(E)} \le \exp \bigg(- \dfrac {\alpha \rho_{t-s}(E, F)^2}{t-s}\bigg) \|\psi\|_{\L^2_{x,v}(F)}
\end{equation}
where $\alpha= \frac {3\lambda}{104\Lambda^2}$
and for $\tau\ne 0$, $\rho_{\tau}(E,F)\ge 0$ is defined by
$$
\rho_{\tau}(E,F)^2= \inf_{(x,v)\in E, (y,w)\in F} \bigg( \frac{|x-y-  \tau w|^2}{\tau^2}+ |v-w|^2\bigg).
$$
\end{thm}

\begin{proof} 
Without loss of generality, we may assume $s=0$. 

 Fix $\tau>0$ and the closed sets $E,F$. For $\psi$ with ${\rm supp}\,\psi \subset F$, we want to estimate $\|\Gamma(\tau,0)\psi\|_{\L^2_{x,v}(E)}$.
To do this, we import an idea going back to Gaffney \cite{MR0102097} and developed by B. Davies \cite{MR1346221} for parabolic equations to the kinetic context.  The idea is to first derive a general inequality for perturbed solutions by exponentials and then optimize to get the desired decay. Here, the exponentials must respect the characteristics of the transport field; in fact,  the exponentials of Davies were constant in time, thus vanishing by applying the time derivative field.

Let $h:\R^{2d}\to \R$ be a bounded and Lipschitz function to be chosen later. For $t\ge 0$, set $h_{t}(x,v)=h(x-tv,v)$ and 
 $f_{h}(t)=\e^{h_{t}}f(t)$, that is, $f_{h}(t)(x,v) = \e^{h_{t}(x,v)}f(t)(x,v)$,  where $f$ at time $t$ is given by $f(t)=\Gamma (t,0)(\e^{-h}\psi)$. Write  $f_{2h}(t)=\e^{2h_{t}}f(t)$.
Using the absolute continuity of $t\mapsto \|f(t)\|_{\L^2_{x,v}}^2$ together with   $(\partial_{t}+v\cdot\nabla_{x})(h(x-tv,v))=0$  and the fact that $ \|f_{h}(t)\|_{\L^2_{x,v}}^2= \angle{f(t)}{f_{2h}(t)}$, we have for almost every $t>0$ (the brackets are the inner product in $(\L^2_{x,v})^d$),  writing $f_{h}$ for $f_{h}(t)$, and $f$ for $f(t)$,
\begin{align*}
 \frac{\dd}{\dd t} \|f_{h}\|_{\L^2_{x,v}}^2   
& = -  2\,\Re \angle {\mathbf{A}\nabla_{v} f}{\nabla_{v} (\e^{2h_{t}} f)}\,  
\\
& = -  2\,\Re \angle {\mathbf{A}\nabla_{v} f}{\e^{2h_{t}}\nabla_{v}  f} - 4\,\Re \angle {\mathbf{A}\nabla_{v} f}{ f \e^{2h_{t}} \nabla_{v}h_{t} } 
\\
& \le -2\lambda \|\e^{h_{t}} \nabla_{v} f\|_{\L^2_{x,v}}^2 +4\Lambda \|\nabla_{v}h_{t}\|_{\infty}  \|\e^{h_{t}}\nabla_{v} f\|_{\L^2_{x,v}}  \|  f_{h}\|_{\L^2_{x,v}}
\\
& \le \frac{2\Lambda^2}{\lambda}\|\nabla_{v}h_{t}\|_{\infty}^2 \|f_{h}\|_{\L^2_{x,v}}^2
\\
&
\le \frac{2\Lambda^2}{\lambda}(\|\nabla_{v}h\|_{\infty}+ t \|\nabla_{x}h\|_{\infty})^2\ \|f_{h}\|_{\L^2_{x,v}}^2
\\
&
\le \frac{4\Lambda^2}{\lambda}(\|\nabla_{v}h\|_{\infty}^2 + \|\nabla_{x}h\|_{\infty}^2 t^2) \ \|f_{h}\|_{\L^2_{x,v}}^2.
\end{align*} 
As $f_{h}(t)\to \psi $ in $\L^2_{x,v}$ when $t\to 0$, we get from Gronwall's lemma for all $t>0$ {and taking square roots}, 
\begin{align}
\label{eq:hperturbation}
\|f_{h}(t)\|_{\L^2_{x,v}} \le  \e^{\kappa  (\|\nabla_{v}h\|_{\infty}^2t + \|\nabla_{x}h\|_{\infty}^2\frac{t^3}3)}\  \|\psi\|_{\L^2_{x,v}} 
\quad \text{ with } \quad  \kappa=\frac{2\Lambda^2}{\lambda}.
\end{align}
 We now make the following choice for $h$. Recall that $\tau >0$,  $E$ and $F$ are fixed.  We  pick  $h\ge 0$  defined by 
 $$h(x,v)^2=\min \bigg(\delta^2  \min_{ (y,w)\in F} \bigg(\frac{|x-y {+}   \tau(v-w)|^2}{\tau^2}+ |v-w|^2 \bigg), N^2\bigg), \qquad (x,v)\in \R^{2d},$$
where  $\delta>0$ is to be chosen and     
with a large $N>\delta  \rho_{\tau}(E,F)$. We can see that $h=0$ on $F$, and ${h_{\tau}}\ge \delta  \rho_{\tau}(E,F)$ on $E$. Moreover, $h$ is bounded by $N$ and Lipschitz with $\|\nabla_{v}h\|_{\infty} \le 2\delta$ and $\|\nabla_{x}h\|_{\infty} \le \delta/\tau$.  As  $\e^{-h}\psi=\psi$ because the support of $\psi$ is contained in $F$, we have {$\Gamma(t,0)\psi=f(t)=\e^{-h_t} f_h (t)$,} and  
we obtain by \eqref{eq:hperturbation} at $t=\tau$ {and the lower bound on $h_{\tau}$ on $E$}, 
\[
\|\Gamma(\tau,0)\psi\|_{\L^2_{x,v}(E)}\le \e^{-\delta \rho_{\tau}(E,F)} \ \|f_{h}(\tau)\|_{\L^2_{x,v}}  \le \e^{-\delta \rho_{\tau}(E,F) +  \kappa \delta^2 \left(4+\frac 1 3 \right) \tau}\ \|\psi\|_{\L^2_{x,v}(F)}.
\]
We still have the choice of $\delta$, so that for 
$\delta= \frac {3\rho_{\tau}(E,F) }{26\kappa \tau}$  we obtain
\[
\|\Gamma(\tau,0)\psi\|_{\L^2_{x,v}(E)}\le \exp({-\tfrac {3\rho_{\tau}(E,F)^2 }{52\kappa \tau}})\|\psi\|_{\L^2_{x,v}(F)}. \qedhere
\] 
\end{proof}

\begin{rem}
 This argument does not try to optimize the expression of the constant $\alpha$. Its dependency on ellipticity constants may be improved for self-adjoint matrices by employing the inequality $2|\angle{a}{b}| \le \angle a a^2+\angle b b ^2$  for the $\mathbf{A}$ induced scalar product. In this case, $\Lambda^2/{\lambda}$ may be replaced by $\Lambda$. 
\end{rem}

\begin{rem}
 It is interesting (and very natural, of course) to see the Galilean group law showing up in the expression of $\rho_\tau(E,F)$ in the proof. The quantity $\rho_{\tau}(E,F)$ is not symmetric in $E$ and $F$. Symmetrizing imposes a change of sign for $\tau$. Indeed, $\rho_{\tau}(E,F)\sim \rho_{-\tau}(F,E)$ as one can easily check.
\end{rem}

Either by the same proof or the symmetry observation above, we also have a bound for the adjoint fundamental solution $\Gamma(t,s)^*=\widetilde \Gamma(s,t)$.
\begin{cor}
With  the  notation of Theorem~\ref{thm:fundamental-bounds}
 we have that 
 for all closed sets $E,F$ in $\R^{2d}$, $s<t$ and $\tpsi\in \L^2_{x,v}$, supported in $F$, 
\begin{equation}
\label{eq:L^2decay*}
\|\widetilde\Gamma(s,t) \tpsi\|_{\L^2_{x,v}(E)} \le \exp \bigg(- \dfrac {\tilde\alpha \rho_{s-t}(E, F)^2}{t-s}\bigg) \|\tpsi\|_{\L^2_{x,v}(F)}.
\end{equation}

\end{cor}

\subsection{Conservation property}

We establish in this section the conservation property saying that the fundamental solution operators preserve constants. At this level of generality, the decay estimate shows that it maps bounded functions into locally square integrable functions.

\begin{thm}[Conservation property]
 \label{thm:conservationproperty}
 Under the assumptions of Theorem~\ref{thm:fundamental-bounds}
 we have  for all $t>s$, 
 $\Gamma(t,s)1=1$ and $\widetilde \Gamma (s,t)1=1$ where $1$ is the function taking the constant value 1, and the equalities are taken in $\L^2_{\loc}(\R^{2d})$ sense. 
\end{thm}

We need the following lemma, whose proof is postponed.

\begin{lem}
\label{lem:I}
 For $R>0$, let $B_{R}=B(0,R)$ be the ball of radius $R$ centered at 0 in $\R^{d}$ and \pascal{$\chi_{R}(x,v)=\chi(x/R,z/R)$ where $\chi$ is a smooth real-valued function  that is 1 on $B_{1}\times B_{1}$ and supported in $B_{2}\times B_{2}$}.  For any fixed $\psi\in \L^2_{x,v}$ with compact support, we have
\begin{enumerate}
\item \pascal{For all $c>0$ there exists a constant $C<\infty$, depending on $c$ and $\psi$, such that  for $0<t-s<c$,}  
\[\iint_{\R^{2d}} | (\Gamma (t,s)\psi)(x,v)|\, \dx\dv \le C.\]
\item For any $s,t,t'$ with $ s<t<t'$, 
\[\lim_{R\to \infty} \iint_{\R^{2d}} (( \Gamma (t',s)\psi)(x,v)- ( \Gamma (t,s)\psi)(x,v)) \chi_{R}(x,v) \dx\dv =0. \]
\item For any fixed $R$ large enough, 
\[\lim_{t\to s^+} \iint_{\R^{2d}\setminus B_{R}\times B_{R}} | (\Gamma (t,s)\psi)(x,v)|\, \dx\dv =0.\] 
 \item For any fixed $R$ large enough, 
\[\lim_{t\to s^+} \iint_{B_{R}\times B_{R}} |\psi(x,v)- (\Gamma (t,s)\psi)(x,v)|\, \dx\dv =0.\]
\end{enumerate}
\end{lem}

\begin{proof}[Proof of Theorem~\ref{thm:conservationproperty}] The proof is the same in both cases and we only prove $\widetilde \Gamma (s,t)1=1$. It amounts to showing that  when $\psi\in \L^2_{x,v}$ with compact support,
\begin{align*}
 \iint_{\R^{2d}} ( \Gamma (t,s)\psi)(x,v)\, \dx\dv =  \iint_{\R^{2d}} \psi(x,v)\, \dx\dv.
\end{align*}
We use the estimates and notation in the previous lemma.
Set
 \begin{align*}
 I&=\iint_{\R^{2d}} \psi(x,v)\, \dx\dv, \\
 I_{t}&=\iint_{\R^{2d}} ( \Gamma (t,s)\psi)(x,v) \dx\dv, \\
 I_{R,t}&= \iint_{\R^{2d}} ( \Gamma (t,s)\psi)(x,v) \chi_{R}(x,v) \dx\dv.
 \end{align*} The integrals exist for all $R>0$ and $t>s$ thanks to (i) and $I_{R,t}\to I_{t}$ as $R\to \infty$ by  the dominated convergence theorem. By (ii), we have $I_{R,t}-I_{R,t'}\to 0$ as $R\to \infty$, thus $I_{t}=I_{t'}$ for all $t,t'\in (s,\infty)$. It remains to show that  $I_{t}\to I$ as $t\to s^+$ because it will imply $I_{t}=I$ for all $t>s$. Now fix $R$ large enough so that the support of $\psi$ is contained in $B_{R}\times B_{R}$.  For $t>s$,  one easily obtains
\begin{align*}
|I_{t}- I| &\le  \iint_{\R^{2d}\setminus B_{R}\times B_{R}} | (\Gamma (t,s)\psi)(x,v)|\, \dx\dv \\
&+ 
\iint_{B_{R}\times B_{R}} |\psi(x,v)- (\Gamma (t,s)\psi)(x,v)|\, \dx\dv
\end{align*}
so that by (iii) and (iv), we conclude that the $I_{t}\to I$ as desired. 
\end{proof}

To prove Lemma~\ref{lem:I}, we need two preliminary estimates, the first one being the classical Caccioppoli inequality (or local energy estimate).
\begin{lem}[Local energy estimates and decay estimates]
\label{lem:Caccioppoli}  With the notation $B_{R}$ as in Lemma~\ref{lem:I} we have
\begin{enumerate}
\item Let  $\psi\in \L^2_{x,v}$. For any real-valued $\chi\in \cD(\R^2)$ and any $s<t<t'$, and with $S_{t,t'}=[t,t']\times \R^{2d}$,
\begin{align*}
  \iiint_{S_{t,t'}} |\nabla_{v}&(\chi\Gamma (\tau,s)\psi)(x,v)|^2\, \dx\dv\dd\tau \le  \frac{2}\lambda \iint_{\R^{2d}} | (\chi \Gamma (t,s)\psi)(x,v)|^2\, \dx\dv\\
 & + C_{\lambda,\Lambda}  \iiint_{S_{t,t'}} [(|\nabla_{v}\chi|^2 + |\chi (v\cdot \nabla_{x}\chi)|)(\Gamma (\tau,s)\psi)](x,v)|^2\, \dx\dv\dd\tau.
 \end{align*}

\item Let  $\psi\in \L^2_{x,v}$ with compact support. For all $0<c<\infty$, there is $\alpha'>0$ and $C<\infty$ (depending on $\psi$ also) such that for $0<t-s<c$ and for $R$ large enough,  
\begin{equation}
\label{eq:(ii)}
\iint_{\R^{2d}\setminus B_{R}\times B_{R}} | (\Gamma (t,s)\psi)(x,v)|^2\, \dx\dv \le C \sup_{b=1,3}(t-s)^{\frac{bd}{2}} \e^{-\alpha'  \frac{R^2}{ (t-s)^{b}}}.
\end{equation}  
\end{enumerate} 
\end{lem}
\begin{proof}
To prove (i), we set $f(t)=\Gamma (t,s)\psi$ and we use the absolute continuity of $t\mapsto \|\chi f(t)\|_{\L^2_{x,v}}^2$ for $t>s$ since $f$ is a weak solution to $ (\partial_t + v \cdot \nabla_x ) f +\cA f = 0$ on $(s,\infty)\times \R^{2d}$ and using the equation, we have
\begin{align*}
   \|\chi f(t')\|_{\L^2_{x,v}}^2   & -  \|\chi f(t)\|_{\L^2_{x,v}}^2  \\
   &= 2  \int_{t}^{t'} \angle{ (\partial_t + v \cdot \nabla_x ) (\chi f)}{\chi f} \dd\tau  \\
                                   & = 2  \iiint_{S_{t,t'}} \chi (v\cdot \nabla_{x}\chi) |f|^2  - 2 \Re \int_{t}^{t'} \angle{ \mathbf{A}\nabla_{v}f}{\nabla_v (\chi^2 f)} \dd\tau \\
                                   & = 2  \iiint_{S_{t,t'}} \chi (v\cdot \nabla_{x}\chi) |f|^2  - 2 \Re \int_{t}^{t'} \angle{\mathbf{A} \chi \nabla_{v}f}{f \nabla_v \chi +  \nabla_v (\chi f)} \dd\tau \\
                                   & = 2  \iiint_{S_{t,t'}} \chi (v\cdot \nabla_{x}\chi) |f|^2  \\
  &- 2 \Re \int_{t}^{t'} \angle{\mathbf{A}  (\nabla_{v}(\chi f) - f \nabla_v \chi)}{f \nabla_v \chi +  \nabla_v (\chi f)} \dd\tau  \\
                                   & = 2  \iiint_{S_{t,t'}} \chi (v\cdot \nabla_{x}\chi) |f|^2  - 2 \Re \int_{t}^{t'} \angle{\mathbf{A}  \nabla_{v}(\chi f) }{\nabla_v (\chi f)} \dd\tau\\
  &- 2 \Re \int_{t}^{t'} \angle{\mathbf{A}  \nabla_{v}(\chi f) }{f \nabla_v \chi} \dd\tau  + 2 \Re \int_{t}^{t'} \angle{\mathbf{A}  f \nabla_v \chi}{f \nabla_v \chi} \dd\tau.
\end{align*}
We use next the ellipticity of $\mathbf{A}$ in order to get,
\begin{align*}
   \|\chi f(t')\|_{\L^2_{x,v}}^2    -  \|\chi f(t)\|_{\L^2_{x,v}}^2  
     & \le 2  \iiint_{S_{t,t'}} |\chi (v\cdot \nabla_{x}\chi)| |f|^2  - 2 \lambda  \iiint_{S_{t,t'}}   |\nabla_{v}(\chi f)|^2  \\
  & + 2 \Lambda \iiint_{S_{t,t'}}  |\nabla_{v}(\chi f)| |f \nabla_v \chi|   + 2 \Lambda \iiint_{S_{t,t'}}  f^2 |\nabla_v \chi|^2.
\end{align*}
It is now easy to conclude the proof of (i).

To prove (ii), we argue using coverings as follows. Assume that the support of $\psi$ is contained in $F=B_{r}\times B_{r}$. Then remark that for 
$(y,w)\in B_{r}\times B_{r}$ and $0<t-s<c$ then $(y+(t-s)w, w) \in B_{ar}\times B_{ar}$ with $a=1+c$. Take a covering of $\R^d$ by the balls $B(j)=B({j},\sqrt d)$, $j\in \IZ^d$. 
Observe that  for $(j,k)\in \IZ^{2d}$,
\[\rho_{t-s}( B(j)\times B(k), F) \ge \frac{ \rho(B(j), B_{ar})^2}{(t-s)^2} + \rho(B(k),B_{ar})^2
\]
where $\rho$ stands for the distance induced on sets by the Euclidean norm. 
If $R$ is large enough, we can  use \eqref{eq:L^2decay} and get,
\begin{align*}
\iint_{\R^{2d}\setminus B_{R}\times B_{R}}& | (\Gamma (t,s)\psi)(x,v)|^2\, \dx\dv \le \sum_{|k|\ge R-\sqrt d \ \mathrm{or}\ |j|\ge R-\sqrt d} \|\Gamma (t,s)\psi\|_{\L^2_{x,v}(B(j)\times B(k))}^2
\\
&
\le \sum_{|k|\ge R-\sqrt d \ \mathrm{or}\ |j|\ge R-\sqrt d} \e^{-2\alpha\frac{ \rho(B(k), B_{ar})^2}{(t-s)^3} -2\alpha \frac{\rho(B(j),B_{ar})^2}{t-s}}\|\psi\|_{2}^2.
\end{align*}
Remark that when $|j|  \ge 2 (\sqrt d + a r)$, 
\begin{equation}
\rho(B(j), B_{ar}) \ge |j|- \sqrt d - ar \ge |j|/2.
\end{equation}
Hence, with $b=1$ or $3$, one can see when $t-s<c$ that  
\begin{equation}
\label{eq:sum1} \sum_{j\in \IZ^d} \e^{-2\alpha \frac{\rho(B(j),B_{ar})^2}{(t-s)^b}}  \lesssim 1\end{equation}
and,  when $R>\sqrt d+  2 (\sqrt d + a r)$,
\begin{equation}
\label{eq:sum2}\sum_{|j|\ge R-\sqrt d} \e^{-2\alpha \frac{\rho(B(j),B_{ar})^2}{(t-s)^b}}  \le \sum_{|j|\ge R-\sqrt d} \e^{-2\alpha \frac{|j|^2}{4(t-s)^b}} \lesssim (t-s)^{\frac{bd}{2}} \e^{-\alpha' \frac{R^2}{(t-s)^{b}}}, 
\end{equation}
where the implicit constants  depend on $c$ and $r$, hence on the support of $\psi$, but not on $R$. 
Thus, splitting the above double sum in $(j,k)$ in two parts and using these bounds, we obtain the right-hand side of \eqref{eq:(ii)}. 
  \end{proof}

\begin{proof}[Proof of Lemma~\ref{lem:I}]

For (i), we use the covering in the proof above and  by Cauchy Schwarz inequality, we have
\[\iint_{\R^{2d}} | (\Gamma (t,s)\psi)(x,v)|\, \dx\dv \le \sum_{j,k} b_{d} \|\Gamma (t,s)\psi\|_{\L^2_{x,v}(B(j)\times B(k))},\]
where $b_{d}$ is the square root of volume of $B(j)\times B(k)$, which is independent of $j,k$. Now we use again \eqref{eq:L^2decay} and argue 
using \eqref{eq:sum1} for each sum in $j$ and $k$.

For (iii), we argue as in the Lemma above with \eqref{eq:sum1} and  \eqref{eq:sum2} for the tail, which gives us a term tending to 0 as $t-s\to 0$.

The proof of (iv) is a consequence of the continuity of $t\in [s,\infty)\to \Gamma(t,s)\psi$ in $\L^2_{x,v}$ and Cauchy-Schwarz inequality. 

We turn to the proof of (ii). We fix $t,t'$ with $s<t<t'$. Write $f(t)=\Gamma (t,s)\psi$. Remark that  the integrals rewrite $\angle{f (t')}{\chi_{R}}- \angle{f(t)}{\chi_{R}}$ and we use the absolute  continuity to obtain that this is equal to
\begin{align*}
& \int_{t}^{t'} \angle{ (\partial_t + v \cdot \nabla_x )  f}{\chi_{R} } \dd\tau  + 
 \int_{t}^{t'} \angle{ f}{ (\partial_t + v \cdot \nabla_x )  \chi_{R}} \dd\tau
 \\
 & =  -  \int_{t}^{t'} \angle{ A \nabla_v   f}{\nabla_{v}\chi_{R} } \dd\tau  + 
 \iiint_{S_{t,t'}}  f v \cdot \nabla_x   \chi_{R}.
\end{align*}
We notice that both  $\nabla_{v}\chi_{R}$ and  $ v \cdot \nabla_x   \chi_{R}$ are supported in  $B_{2R}\times B_{2R}\setminus B_{R}\times B_{R}$. Thus the right-hand side is unchanged if  we multiply  $f$
by a function $\tilde \chi_{R}$ where $\tilde\chi$ is supported in  
$B_{4}\times B_{4}\setminus B_{1/2}\times B_{1/2}$ and is one on $B_{2}\times B_{2}\setminus B_{1}\times B_{1}$. 
Both integrals are controlled using Cauchy-Schwarz inequality.   For the one involving $\nabla_{v}(\tilde \chi_{R} f)$  we use  the Caccioppoli inequality in (i) of Lemma~\ref{lem:Caccioppoli} with $\chi$ replaced by $\tilde\chi_{R}$, and then insert the estimate \eqref{eq:(ii)} to conclude when $R\to \infty$. For the second one, we use \eqref{eq:(ii)} directly. Further details are left to the reader.  
\end{proof}

\subsection{Local weak solutions}
\label{sec:weaksol}

We say that $f$ is a local weak solution to the kinetic equation $ (\partial_t + v \cdot \nabla_x ) f +\cA f = 0$ in an open set $\Omega \subset \R^{2d+1}$, if   
$f$ and  $\nabla_{v}f$   belong to  $\L^2(\Omega)$ and  the  equation is satisfied weakly against test functions $h \in \cD(\Omega)$:
\[
  -\iiint_{\Omega} f (\partial_{t}+v\cdot\nabla_{x})  h  \dt \dx \dv + \iiint_{\Omega} \mathbf{A} \nabla_{}{v} f \cdot {\nabla_{v}h} \dt \dx \dv= 0. 
\]

We say that $f$ is a local weak solution to the backward kinetic equation $ -(\partial_t + v \cdot \nabla_x ) f +\cA^* f = 0$ in an open set $\Omega$, if   
$f$ and  $\nabla_{v}f$   belong to  $\L^2(\Omega)$ and  the  equation is satisfied weakly against test functions $h \in \cD(\Omega)$:
\[
  +\iiint_{\Omega} f (\partial_{t}+v\cdot\nabla_{x})  h  \dt \dx \dv + \iiint_{\I\times \U_{x}\times \U_{v}} \mathbf{A}^* \nabla_{v} f \cdot {\nabla_{v}h} \dt \dx \dv= 0. 
\]

\begin{lem}[Continuity in $\L^2_{\loc}$ of local weak solutions]\label{lem:continuity} Let    $\I\times \U_{x}\times \U_{v}$ be an open set whose closure is a subset  of $ \Omega$. Any local weak solution $f$ on $\Omega$ as above belongs to $\C(\overline \I\  ;\,  \L^2(\U_{x}\times \U_{v}))$.
\end{lem}

\begin{proof} Define $\tf$ on $\R\times\R^d\times \R^d$ as the zero extension outside of $\Omega$ of $\tf=f\chi$, where $\chi$ is a smooth cut-off localized in a neighborhood of the closure of $\I\times \U_{x}\times \U_{v} $ contained in $\Omega$. Then $\tf \in \L^2_{t,x}\W^{1,2}_{v}$  and satisfies an equation $(\partial_t + v \cdot \nabla_x ) \tf =-\cA \tf + \div G +g $ using the canonical extension of $\cA$ and $G,g$ are both $\L^2_{t,x,v}$ functions. Thus $\tf$ belongs to the space $\cF^\beta_{\beta}$ of \cite{AIN}. Using  $\cF^\beta_{\beta}\subset \cL^\beta_{\beta}$ and Theorem 1.1 there, we have that $\tf \in \C_{0}(\R\, ;\, \L^2_{x,v}).$ The conclusion follows by restriction.
 \end{proof}

{The following definition introduces quantitative boundedness for local weak solutions.}

For a point $z_{0}=(t_{0}, x_{0},v_{0})\in \R\times \R^d\times \R^d$, and $r>0$, the forward  and backward kinetic cylinders centered at $z_{0}$ and of size $r$ are defined by
\begin{align*}
Q_{r}(z_{0})=\{(t,x,v)\, ; \, t\in (t_{0}-r^2, t_{0}], |v-v_{0}|<r, |x-x_{0}-(t-t_{0})v_{0}| <r^3\}
\end{align*}
\begin{align*}
Q_{r}^*(z_{0})=\{(t,x,v)\, ; \, t\in [t_{0}, t_{0}+r^2), |v-v_{0}|<r, |x-x_{0}\  {-}\  (t-t_{0})v_{0}| <r^3\}
\end{align*}
Let 
$B_{r}(x_{0},v_{0})=\{(x,v)\, ; \,  |v-v_{0}|<r, |x-x_{0}| <r^3\}$.

\begin{defn}[Local Boundedness property] 
\label{def:localboundedness}
 We say that the operator  $ (\partial_t + v \cdot \nabla_x )  +\cA  $, respectively   $ -(\partial_t + v \cdot \nabla_x )  +\cA^* $,  has the local boundedness property if there exists $0<B<\infty$ such that for all $z_{0}\in \R^{2d+1}$ and $r>0$, any local weak solution to 
 $ (\partial_t + v \cdot \nabla_x ) f +\cA f = 0$, in an open neighborhood of $Q_{2r}(z_{0})$, respectively $ -(\partial_t + v \cdot \nabla_x ) \tilde f +\cA^* \tilde f = 0$ and $Q_{2r}^*(z_{0})$,  has local bounds of the form, respectively, 
\begin{align}\label{eq:localbound}
\esssup_{B_{r}(x_{0},v_{0})}|f(t_{0},\cdot)|^2 &\le \frac {B^2} {r^{4d+2}} \iiint_{Q_{2r}(z_{0})} |f|^2\, \dd t\dd x \dd v ,
 \\
\label{eq:localbound*}  \esssup_{B_{r}(x_{0},v_{0})}|\tilde f(t_{0},\cdot)|^2 &\le \frac {B^2} {r^{4d+2}} \iiint_{Q_{2r}^*(z_{0})} |\tilde f|^2 \, \dd t\dd x \dd v .
\end{align}
\end{defn}

\begin{rem}
 The condition is scale invariant: For a point $z_{0}=(t_{0}, x_{0},v_{0})\in \R\times \R^d\times \R^d$, let 
$\cT_{z_{0}}$ be the transformation 
\begin{align*}
 \cT_{z_{0}}(t,x,v)=(t_{0}+t, x_{0}+x+tv_{0}, v_{0}+v)
\end{align*}
and for $r>0$,  $\delta_{r}(t,x,v)= (r^2t, r^3 x, r v)$. Then $Q_{r}(z_{0})= \cT_{z_{0}}(\delta_{r}(Q_{1}(0,0,0))).$
\end{rem}

\begin{rem} Note that these conditions are usually presented by taking suprema on $Q_{r}(z_{0})$, $Q_{r}^*(z_{0})$ respectively, which means that one needs to know that local weak solutions have pointwise values. Our weaker formulation where we take the essential supremum only on the ``top'' and ``bottom'' of $Q_{r}(z_{0})$ and $Q_{r}^*(z_{0})$ respectively suffices. \pascal{As a matter of fact,  since solutions are continuously valued into $\L^2_{\loc}$ (Lemma \ref{lem:continuity}), this formulation is equivalent to scale invariant  $\L^2-\L^\infty$ Moser estimates where one takes the essential suprema on $Q_{r}(z_{0})$, $Q_{r}^*(z_{0})$ (for all $r, z_{0}$)  on the left hand sides, respectively. In particular, this criterion does not require to know that solutions {are} continuous, or even defined, at all points. } \end{rem}

\subsection{Gaussian upper bound and scale invariant $\L^2-\L^\infty$ Moser estimates}

\begin{thm}[Pointwise Gaussian upper bound for the fundamental solution]
\label{thm:MoserimpliesGUB}
 Assume the conditions of Theorem~\ref{thm:fundamental-bounds} and that $ (\partial_t + v \cdot \nabla_x )  +\cA  $ and   $ -(\partial_t + v \cdot \nabla_x )  +\cA^* $  have the local boundedness property  \eqref{eq:localbound} and \eqref{eq:localbound*}.
  
Then, for all $t>s$, the fundamental solution operator $\Gamma(t,s)$ has  an integral kernel  $\Gamma(t,x,v,s,y,w)$, called generalized fundamental solution, in the sense that  for all $\psi\in \L^2_{x,v}$, for almost every $(x,v) \in \R^{2d}$, 
\begin{equation}
\label{eq:integralrepresentation}
(\Gamma(t,s)\psi)(x,v)=\iint_{\R^{2d}} \Gamma(t,x,v,s,y,w)\psi(y,w)\, \dd y \dd w, 
\end{equation}
with  pointwise kinetic Gaussian upper bound for almost every $(x,v,y,w) \in \R^{4d}$,
\begin{equation}
\label{eq:GUB}
|\Gamma(t,x,v,s,y,w)|\le \frac {C_{d,\lambda,\Lambda}B^2} {(t-s)^{2d}}\e^{-c_{\lambda,\Lambda}\big({\tfrac{|x-y-  (t-s) w|^2}{(t-s)^3}+ \tfrac{|v-w|^2}{t-s}}\big)},
\end{equation}
with $c_{\lambda,\Lambda}>0$. 
 \end{thm}

 \begin{proof}  Let us introduce the quantity
 \[\theta(h,a)=\|\nabla_{v}h\|_{\infty}\,a^{1/2} + 4\|\nabla_{x}h\|_{\infty}\, a^{3/2}, 
\]
when $a>0$ and $h$ is a function on $\R^{2d}$. 

Under the hypotheses of Theorem~\ref{thm:fundamental-bounds} we have proved \eqref{eq:hperturbation}, which implies that for all $t>s$, $\psi\in \L^2_{x,v}$ and real, Lipschitz and bounded function  $h$ on $\R^{2d}$, 
\begin{align*}
\|f_{h}(t)\|_{\L^2_{x,v}} \le  \e^{\kappa  \theta(h,t-s)^2}\  \|\psi\|_{\L^2_{x,v}}.
\end{align*}
where $\kappa=\frac{2\Lambda^2}{\lambda}$, with $f_{h}(t)
= \e^{h_{t-s}}\Gamma(t,s)(\e^{-h}\psi)$.  

Step 1: We may apply \eqref{eq:localbound} to $f(t)=\Gamma(t,s)(\e^{-h}\psi)$ and obtain for $0<t-s$ and $(x,v)\in \R^{2d}$ that for almost every $(x',v')\in B_{\sqrt {t-s}/4}(x,v)$,
 \begin{align*}
 |f(t,x',v')|^2  &\le \frac {2^{2+4d}B^2} {(t-s)^{1+2d}} \iiint_{Q_{\sqrt{t-s}/2}(t,x,v)} |f(\tau,y,w)|^2 \, \dd y\dd w \dd \tau. 
 \end{align*}
 Thus, 
\begin{align*}
 |f_{h}(t,x',v')|^2  &\le \frac {2^{2+4d}B^2} {(t-s)^{1+2d}} \iiint_{Q_{\sqrt{t-s}/2}(t,x,v)} \e^{2h_{t-s}(x',v')-2h_{\tau-s}(y,w)}|f_{h}(\tau,y,w)|^2 \, \dd y\dd w \dd \tau. 
\end{align*}
Using the Lipschitz bounds for $h$,
\begin{align*}
 h_{t-s}(x',v')-&h_{\tau-s}(y,w)= h(x'-(t-s)v',v')-h(y-(\tau-s)w, w)
 \\
 & \le \|\nabla_{v}h\|_{\infty}|v'-w| + \|\nabla_{x}h\|_{\infty}|x'-(t-s)v'-(y-(\tau-s)w)|.
\end{align*}
Using $(x',v')\in B_{\sqrt {t-s}/4}(x,v)$ and $(\tau,y,w)\in Q_{\sqrt{t-s}/2}(t,x,v)$,
\begin{align*}
 |x'-(t-s)v'&-(y-(\tau-s)w)| \\
 & \le |x'-x|+ |y-x-(\tau-t)v| + |-(t-s)v'-(\tau-t)v+(\tau-s)w|
 \\
 & \le |x'-x|+ |y-x-(\tau-t)v| + |(t-s)(v-v')|+|(\tau-s)(w-v)|
 \\
 & \le 2(t-s)^{3/2},
\end{align*}
and we obtain 
\[ 2h_{t-s}(x',v')-2h_{\tau-s}(y,w) \le \theta(h,t-s).\]
Thus, 
\begin{align*}
 |f_{h}(t,x',v')|^2  &\le \frac {2^{2+4d}B^2 \e^{\theta(h,t-s)}} {(t-s)^{1+2d}} \iiint_{Q_{\sqrt{t-s}/2}(t,x,v)} |f_{h}(\tau,y,w)|^2 \, \dd y\dd w \dd \tau. 
\end{align*}
For the last integral, we use our $\L^2$ estimate:
\begin{align*}
 \iiint_{Q_{\sqrt{t-s}/2}(t,x,v)} |f_{h}(\tau,y,w)|^2 \, \dd y\dd w \dd \tau 
& \le \int_{\frac{t+s}{2}}^t \|f_{h}(\tau)\|_{\L^2_{x,v}}^2\, \dd \tau
 \\
 &\le \int_{\frac{t+s}{2}}^t \e^{2\kappa  \theta(h,\tau-s)^2}\  \|\psi\|_{\L^2_{x,v}}^2\, \dd\tau
 \\
&\le  \frac{t-s}2\, \e^{2\kappa  \theta(h,t-s)^2}\  \|\psi\|_{\L^2_{x,v}}^2.
\end{align*}
As $(x,v)$ was an arbitrary point in $\R^{2d}$ and $(x',v')$ almost every point in $B_{\sqrt {t-s}/4}(x,v)$, we have obtained
\begin{align*}
\|f_{h}(t)\|_{\L^\infty_{x,v}}^2\le \frac {2^{3+4d}B^2 \e^{\theta(h,t-s)+2\kappa  \theta(h,t-s)^2}} {(t-s)^{2d}} \  \|\psi\|_{\L^2_{x,v}}^2.
\end{align*}
Note that $\psi$ is arbitrary, so this amounts to an $\L^2-\L^\infty$ estimate for the operator $\e^{h_{t-s}}\Gamma(t,s)\e^{-h}$, $t>s$.

Step 2:
For the backward equation,  with the same argument  we obtain the  same inequality for the operator 
 $\e^{h_{s-t}}\widetilde\Gamma (s,t)\e^{-h}$, $s<t$ (here $t$ is fixed and $s$ varies). As $\widetilde\Gamma (s,t)=\Gamma (t,s)^*$, see Proposition~\ref{prop:propGreenop}, by duality and changing $h$ to $-h$, this implies an $\L^1-\L^2$ estimate for the operator $\e^{h}\Gamma (t,s)  \e^{-h_{s-t}}$ with the same bound (More precisely, this operator originally defined $\L^1\cap \L^2$ extends continuously to $\L^1$ into $\L^2$.) 
 
 Step 3: 
 Now we use the Chapman-Kolmogorov relation with $r=\frac{t+s}2$, see Proposition~\ref{prop:propGreenop}. We shall use several times that $t-r=r-s=\frac{t-s}2$. First,
 \begin{align*}
 \e^{h_{t-r}}\Gamma(t,s)\e^{-h_{r-t}}&=\e^{h_{t-r}}\Gamma(t,r)\Gamma(r,s)\e^{-h_{r-t}}
 \\
 &=\e^{h_{t-r}}\Gamma(t,r)\e^{-h}\e^{h}\Gamma(r,s)\e^{-h_{s-r}}.
 \end{align*}
Thus we obtain an $\L^1-\L^\infty$ operator bound for $ \e^{h_{t-r}}\Gamma(t,s)\e^{-h_{s-r}}$ with norm bounded by 
\begin{align*}
\frac {2^{3+4d}B^2 \e^{\theta(h,t-r)+2\kappa  \theta(h,t-r)^2}} {(t-r)^{2d}} = 
\frac {2^{3+6d}B^2 \e^{\theta(h,t-r)+2\kappa  \theta(h,t-r)^2}} {(t-s)^{2d}} .
 \end{align*}
By the Dunford--Pettis theorem (see  \cite[Theorem 1.3]{MR1253180}),  this amounts to the fact that for all $t>s$, $ \e^{h_{t-r}}\Gamma(t,s)\e^{-h_{s-r}}$ 
  is an integral operator with  measurable kernel having $\L^\infty$ norm equal to the $\L^1-\L^\infty$ operator norm. We have shown that for all $t>s$, $\Gamma(t,s)$ has a  measurable kernel, that we denote $\Gamma(t,x,v,s,y,w):=\Gamma(t,s)(x,v,y,w)$, having an almost everywhere bound
\begin{equation}
\label{eq:Gamma(t,x,v,s,y,w)h}
|\Gamma(t,x,v,s,y,w)|\le \e^{h_{s-r}(y,w)-h_{t-r}(x,v)} \frac {2^{3+6d}B^2 \e^{\theta(h,t-r)+2\kappa  \theta(h,t-r)^2}} {(t-s)^{2d}}=:  \frac {2^{3+6d}B^2 \e^L} {(t-s)^{2d}}
\end{equation}
and \eqref{eq:integralrepresentation} holds for $\psi\in \L^1_{x,v}\cap \L^2_{x,v}$. Recall that $h$ is an arbitrary, real, bounded and Lipschitz function. 
 Taking 
$h=0$ already gives us the on-diagonal pointwise bound
\begin{equation}
\label{eq:Gamma(t,x,v,s,y,w)}
|\Gamma(t,x,v,s,y,w)|\le  \frac {2^{3+6d}B^2 } {(t-s)^{2d}}.
\end{equation}

Step 4: The final step is to obtain the pointwise Gaussian bound \eqref{eq:GUB}, which will allow us to obtain \eqref{eq:integralrepresentation} for all $\psi\in \L^2_{x,v}$ as such a bound is integrable with respect to $(y,w)$. Note that $L$ in \eqref{eq:Gamma(t,x,v,s,y,w)h} depends on $h$ and all variables $(t,x,v,s,y,w)$.  We fix $t>s$ and variables $(x,v,y,w)$ for which the above estimate \eqref{eq:Gamma(t,x,v,s,y,w)h} holds, and select $h$ depending on them to minimize $L$ in the exponential factor $\e^L$.

For $0<\delta, N<\infty$ appropriate,  we let $h$ be the positive function defined by 
\[
h(x',v')^2=\min \bigg(\delta^2   \bigg(\frac{|x'-y -(t-r)(2w-v')|^2}{(t-s)^2}+ |v'-w|^2 \bigg), N^2\bigg), \quad (x',v')\in \R^{2d}.
\]
We see that  $h$ is bounded, and Lipschitz with $\|\nabla_{x'}h\|_{\infty}\le \frac \delta {t-s}$ and $\|\nabla_{v'}h\|_{\infty}\le \frac{3\delta}2$. Moreover, 
\[h_{s-r}(y,w)=h(y-(s-r)w,w)=h(y+(t-r)w,w)= 0
\]
and with $E=\{(x,v)\}$ and $F=\{(y,w)\}$,
\[
h_{t-r}(x,v)=h(x-(t-r)v,v)=  \delta\rho_{t-s}(E,F)
\]
provided $N$ is larger than this last quantity. Moreover, 
\[\theta(h, t-r) \le \frac {3\delta}{2\sqrt 2}(t-s)^{1/2} + \frac {4\delta}{t-s} |t-r|^{3/2} \le 2\sqrt 2 \delta (t-s)^{1/2}.
\]
Now we choose 
\[
\delta= \frac{\rho_{t-s}(E,F)}{32\kappa (t-s)}.
\]
Then, a simple calculation shows  that 
\[
L  \le -\frac{\rho_{t-s}(E,F)^2}{64\kappa (t-s)}+ \frac{\sqrt 2 \rho_{t-s}(E,F)}{16\kappa (t-s)^{1/2}}.  
\]
Either $ \frac{4\sqrt 2 \rho_{t-s}(E,F)}{(t-s)^{1/2}} \le 1$, in which case 
$L\le \frac 1{64\kappa}-\frac{\rho_{t-s}(E,F)^2}{64\kappa (t-s)}$, or $\frac{4\sqrt 2 \rho_{t-s}(E,F)}{(t-s)^{1/2}} \ge 1$, in which case $L \le -\frac{\rho_{t-s}(E,F)^2}{128\kappa (t-s)}.
$ In any case, 
\[
L\le  \frac 1{64\kappa}-\frac{\rho_{t-s}(E,F)^2}{128\kappa (t-s)}.
\]

\end{proof}

\begin{thm}[Pointwise upper bound implies local boundedness property]
 A converse to Theorem~\ref{thm:MoserimpliesGUB} holds, that is, having a pointwise upper bound \eqref{eq:GUB} implies the local boundedness properties  \eqref{eq:localbound} and \eqref{eq:localbound*}. 
 \end{thm}

\begin{proof}
 It is a mere adaptation of the argument in \cite[Theorem 1.2]{MR2091016} using localization, size and local energy (Caccioppoli, see Lemma~\ref{lem:Caccioppoli}) estimates, while Duhamel's formula is justified by Theorem~\ref{thm:representation0T}.  We skip details. 
\end{proof}

\begin{thm}[Properties of the generalized fundamental solution]
\label{thm:propertiesFS}
 Under the assumptions of  Theorem~\ref{thm:MoserimpliesGUB}, let $ \Gamma(t,x,v,s,y,w)$ be the generalized fundamental solution of $ (\partial_t + v \cdot \nabla_x )  +\cA  $ and $\widetilde \Gamma (s,y,w,t,x,w)$ the one of $ -(\partial_t + v \cdot \nabla_x )  +\cA^* $. Then, 
  \begin{enumerate}
\item  For all  $s<r<t$ and almost every $(x_{1},v_{1},x_{2},v_{2})$,
\[\Gamma(t,x_{1},v_{1},s,x_{2}, v_{2})=\iint_{\R^{2d}} \Gamma(t,x_{1},v_{1},r,y,w)\Gamma(r,y,w,s,x_{2},v_{2})\, \dd y \dd w.
\]
\item For all $s<t$, one has for almost every $(x,v,y,w)$,
\[
\widetilde \Gamma (s,y,w,t,x,w)= \overline{\Gamma(t,x,v,s,y,w)}.
\]
\item For all $s<t$ and almost every $(x,v)$, 
\[\iint_{\R^{2d}} \Gamma(t,x,v,s,y,w)\, \dd y \dd w = 1= \iint_{\R^{2d}} \Gamma(t,y,w,s,x,v)\, \dd y \dd w.\]
\item If $\psi$ is continuous and bounded then for almost every $(x,v)$, 
\[
\lim_{t\to s^+}\iint_{\R^{2d}} \Gamma(t,x,v,s,y,w)\psi(y,w)\, \dd y \dd w = \psi(x,v).
\]

\end{enumerate}

 \end{thm}

\begin{proof}
 
 Item (i) is a consequence of the Chapman-Kolmogorov relation in Proposition~\ref{prop:propGreenop}, \eqref{eq:integralrepresentation} and Fubini's theorem using integrability from \eqref{eq:GUB}. 

Item (ii) follows from the adjoint relation in Proposition~\ref{prop:propGreenop} and   \eqref{eq:GUB}.

Item (iii) follows from Theorem~\ref{thm:conservationproperty}, \eqref{eq:integralrepresentation},  \eqref{eq:GUB} and item (ii). 

For item (iv), using points where the upper estimate holds and item (iii), it follows from a bound of the type 
\begin{align*}
\bigg| \iint_{\R^{2d}} &\Gamma(t,x,v,s,y,w)\psi(y,w)\, \dd y \dd w - \psi(x,v)\bigg|
\\
&\le C \sup_{B_{r}(x,v)}|\psi(y,w)-\psi(x,v)| +  h(r/(t-s)^{1/2}, v/(t-s)^{1/2}) \|\psi\|_{\infty},
\end{align*}
where $B_{r}(x,v)$ are the balls introduced before Definition~\ref{def:localboundedness}, 
 $h:(0,\infty)\times \R^d\to (0,\infty)$ is a function that  tends to $0$ at $\infty$ and $$C=\sup_{t>s}\esssup_{x,v}\iint_{\R^{2d}} |\Gamma(t,x,v,s,y,w)|\, \dd y \dd w.$$ Details are easy and left to the reader.
\end{proof}

\begin{thm}[Generalized fundamental solution for real coefficients]
\label{thm:conclusion}
If $\mathbf{A}$ in \eqref{e:ellip-lower} has real, measurable,  coefficients, then the operator $ (\partial_t + v \cdot \nabla_x )  +\cA  $ has a generalized fundamental solution  $\Gamma(t,x,v,s,y,w), t>s, (x,v,y,w) \in \R^{4d}$ satisfying \eqref{eq:GUB},  in the sense that  for any $\psi\in \L^2_{x,v}$, the   weak solution $f$ in the class specified in Section~\ref{sec:review}  to the kinetic Cauchy problem $(\partial_t + v \cdot \nabla_x ) f +\cA f=0$ with initial data $\psi$ at time $s$ is represented for all $t> s$ and almost every $(x,v)\in \R^{2d}$ by 
\begin{equation}
\label{eq:integralrepresentationf}
f(t,x,v)=\iint_{\R^{2d}} \Gamma(t,x,v,s,y,w)\psi(y,w)\, \dd y \dd w. 
\end{equation} 
The function $\Gamma$ is the unique  locally integrable function for which \eqref{eq:integralrepresentationf} holds for all $\psi\in \cD(\R^{2d})$ and all $t>s$ and almost every $(x,v)$. Moreover,  it satisfies the properties of Theorem~\ref{thm:propertiesFS}, and for all $S\in \cD(\R\times\R^d\times\R^d)$, 
\begin{equation}
\label{eq:integralrepresentationS}
 (t,x,v)\mapsto \int_{-\infty}^t\iint_{\R^{2d}} \Gamma(t,x,v,s,y,w)S(s,y,w)\, \dd y \dd w\, \ds
\end{equation} agrees with the  weak solution $\cK_{\cA}\, S$ to $(\partial_t + v \cdot \nabla_x ) f +\cA f=S$ on $\R\times\R^d\times\R^d$.
\end{thm}

\begin{proof}
Applying the Moser estimates from \cite{MR3923847} for weak solutions gives us the conclusions of  Theorem~\ref{thm:MoserimpliesGUB} and Theorem~\ref{thm:propertiesFS}. Strictly speaking, \cite{MR3923847} also assumes that the matrix $\mathbf{A}$ is symmetric but examination of the proof shows it is not
necessary.  Uniqueness with \eqref{eq:integralrepresentationf} follows from $f(t)=\Gamma(t,s)\psi$ proved in Theorem~\ref{thm:representation0T}
 (putting the initial time at $s$). Thus  $\Gamma(t,x,v,s,y,w)$ must be the Schwartz kernel of $\Gamma(t,s)$. As for \eqref{eq:integralrepresentationS}, this also follows from Theorem~\ref{thm:representation0T} or its version on $\R$  in Proposition~\ref{thm:representationRhom}.
  
\end{proof}
\begin{rem}
We could have added lower-order terms with bounded coefficients to $\cA$. The methods of proofs for the $\L^2$ decay and the pointwise decay adapt. Moser's estimates need only to be considered on small scales ($r<r_{0}$) and the estimates come with an extra factor $\e^{\omega(t-s)}$. Except for the conservation property (iii) in Theorem \ref{thm:propertiesFS} which could be lost from having lower order coefficients, the conclusions of Theorem \ref{thm:conclusion} hold with a modification to \eqref{eq:integralrepresentationS}: either we assume $S$ has support in $(0,T)$ and this is the weak solution to the Cauchy problem on $(0,T)$ with zero initial data, or we multiply $\Gamma $ by $e^{-c(t-s)}$, $c> c_{0}$,  in the integrand and this is $\cK_{\cA +c}\, S$.
\end{rem}
\begin{rem}
	The methods of Section \ref{sec:localcase} apply to systems of equations, too, under ellipticity in the form of \eqref{e:ellip-lower}. The
	generalized fundamental solution is then matrix-valued and complex conjugation in property (ii) of Theorem~\ref{thm:propertiesFS} should be replaced by matrix adjunction.    
\end{rem}

\bibliographystyle{plain}

\end{document}